%% file: main.tex
\documentclass[10pt, a4paper, notitlepage]{amsart}
\pdfoutput=1

\usepackage{biblatex}
\input{packages.tex}
\usepackage{hyperref}
\input{style.tex}

\input{notation_general.tex}

\title{Calabi-Yau techniques for Namikawa-Weyl groups}
\author{Jasper van de Kreeke}

\DeclareRobustCommand{\SkipTocEntry}[5]{}

\begin{document}

\begin{abstract}
The field of symplectic singularities aims to build a 21st century Lie theory. A key development is the Namikawa-Weyl group, which generalizes the classical Weyl group of Lie algebras. Another cornerstone is the integration of categorical Calabi-Yau techniques, capturing the rich algebraic structure of these singularities. In this paper, we develop a systematic strategy to determine Namikawa-Weyl groups for representation spaces of Calabi-Yau-2 algebras, leveraging local-to-global functors, symmetry analysis, and $ A_∞ $-methods. Applying this approach to quiver varieties, we carry out the detailed calculations and recover Yaochen Wu's result.
\end{abstract}

\maketitle

\tableofcontents

\input{intro.tex}
\input{prelim_short.tex}
\input{result_short.tex}

\appendix

\input{app_examples/intro.tex}
\input{app_examples/example.tex}

\input{app_functor/intro.tex}
\input{app_functor/honest.tex}
\input{app_functor/construction.tex}
\input{app_functor/stability.tex}

\input{app_excfiber/intro.tex}
\input{app_excfiber/generic_cb.tex}
\input{app_excfiber/generic_all.tex}
\input{app_excfiber/partial_case.tex}

\input{app_symmetry.tex}
\input{app_monodromy.tex}
\input{app_products.tex}

\printbibliography

\bigskip
\begin{tabular}{@{}l@{}}%
    \textsc{Department of mathematics, University of California, Berkeley, USA}\\
    \textit{jasper.kreeke@berkeley.edu}
  \end{tabular}

\end{document}

%% file: packages.tex
\usepackage{bbm}

\usepackage{etoolbox}
\usepackage{tikz}
\usetikzlibrary{calc}
\usetikzlibrary{cd}
\usetikzlibrary{decorations.markings}
\usetikzlibrary{decorations.pathreplacing}
\usetikzlibrary{decorations.pathmorphing}
\usetikzlibrary{decorations.text}
\usetikzlibrary{arrows.meta}
\usetikzlibrary{arrows}
\usetikzlibrary{positioning}

\usepackage{geometry} 
\usepackage{fancyhdr} 
\usepackage{longtable}
\usepackage{subcaption}
\usepackage{enumitem}
\usepackage{amssymb}
\usepackage{amsmath}
\usepackage{stackrel} 
\usepackage{uniinput}
\usepackage{amsthm}
\usepackage{stmaryrd}
\usepackage{mathtools}
\usepackage{fouridx}
\usepackage{xparse}
\usepackage{bm}
\usepackage{aliascnt}
\usepackage{import}

%% file: style.tex
\theoremstyle{definition}
\newtheorem{theorem}{Theorem}
\let\emph\textbf

\newaliascnt{remark}{theorem}
\newaliascnt{definition}{theorem}
\newaliascnt{proposition}{theorem}
\newaliascnt{lemma}{theorem}
\newaliascnt{corollary}{theorem}
\newaliascnt{example}{theorem}
\newaliascnt{convention}{theorem}
\newaliascnt{todo}{theorem}

\newtheorem{remark}[remark]{Remark}
\newtheorem{definition}[definition]{Definition}
\newtheorem{proposition}[proposition]{Proposition}
\newtheorem{lemma}[lemma]{Lemma}

\aliascntresetthe{remark}
\aliascntresetthe{definition}
\aliascntresetthe{proposition}
\aliascntresetthe{lemma}
\aliascntresetthe{corollary}
\aliascntresetthe{example}
\aliascntresetthe{convention}
\aliascntresetthe{todo}

\numberwithin{equation}{section}
\numberwithin{figure}{section}
\numberwithin{table}{section}
\makeatletter\let\c@table\c@figure\makeatother

\numberwithin{theorem}{section}
\numberwithin{remark}{section}
\numberwithin{definition}{section}
\numberwithin{proposition}{section}
\numberwithin{lemma}{section}
\numberwithin{corollary}{section}
\numberwithin{example}{section}
\numberwithin{convention}{section}
\numberwithin{todo}{section}

%% file: notation_general.tex
\newcommand{\id}{\mathrm{id}}

\newcommand{\Hom}{\operatorname{Hom}}

\newcommand{\cat}[1]{\mathcal{#1}}

\newcommand{\Tw}{\operatorname{Tw}}
\newcommand{\Mod}{\operatorname{Mod}}

\def\H{\operatorname{H}}
\newcommand{\HTw}{\operatorname{H}\operatorname{Tw}}

\newcommand{\isoto}{\xrightarrow{\sim}}
\newcommand{\verylongisoto}{\xrightarrow{~~~\sim~~~}}
\newcommand{\verylongmapsto}{\xmapsto{~~~\phantom{\sim}~~~}}

\def\Im{\operatorname{Im}}

\newcommand{\running}{~|~}

\newcommand{\Ext}{\operatorname{Ext}}

\newcommand{\Rep}{\operatorname{Rep}}
\newcommand{\qQ}{\overline{Q}}
\newcommand{\Lie}{\operatorname{Lie}}
\newcommand{\GL}{\operatorname{GL}}
\newcommand{\M}{\mathcal{M}}
\newcommand{\quot}{\sslash}
\newcommand{\simp}{\mathrm{s}}

\newcommand{\soc}{\operatorname{soc}}
\newcommand{\Aut}{\operatorname{Aut}}
\newcommand{\reg}{\mathsf{reg}}
\newcommand{\HP}{\operatorname{HP}}

%% file: intro.tex
\section{Introduction}
Symplectic singularities are singular algebraic varieties together with a symplectic form. In 2008, Namikawa proved that the Poisson deformation theory of a symplectic singularity $ (X, ω) $ is related to the Poisson deformation theory of a symplectic resolution $ (Y, ω) $. The correspondence is characterized by a Weyl group, now known as the Namikawa-Weyl group:
\begin{center}
\begin{tikzpicture}
\path (0, 0) node (A) {Poisson deformations of $ Y $};
\path (8, 0) node (B) {Poisson deformations of $ X $};
\path[draw, <->] ($ (A.east) + (right:0.5) $) to node[midway, above] {Weyl group} ($ (B.west) + (left:0.5) $);
\end{tikzpicture}
\end{center}
In the present paper, we determine the Namikawa-Weyl groups for those quiver varieties $ \M(Q, α) $ that have a symplectic resolution.

Symplectic singularities were introduced by Beauville \cite{Beauville} in 1999 with the aim of unifying the algebraic concept of Goreinstein singularities with the geometric concept of symplectic manifolds. Driven by the mantra of deformation theory that deforming a singular object produces a smooth object, Ginzburg and Namikawa \cite{namikawa-I, namikawa-II} investigated the Poisson deformation theory of symplectic singularities during the years of 2005–08. As it turned out, the base of the universal Poisson deformation of a symplectic singularity $ X $ and the base of the universal Poisson deformation of a symplectic resolution $ Y $ are connected by a ramified Galois covering. The covering's Galois group resembles the classical Weyl groups and is now known as the \emph{Namikawa-Weyl group} $ W $ of $ X $:
\begin{center}
\begin{tikzcd}
\mathcal{X} \arrow[r] \arrow[d] & \mathcal{Y} \arrow[d] \\
ℂ^n \arrow[r, ":W"] & ℂ^n.
\end{tikzcd}
\end{center}
The past 25 years were shaped by the question: Which systems in algebraic geometry and string theory can be modeled by symplectic singularities, and what are their Namikawa-Weyl groups? The results now include a full description of the Namikawa-Weyl in case $ X $ is a nilpotent coadjoint orbit \cite{namikawa-II} or an affine quotient \cite{bellamy-cm}.

Quiver varieties $ X = \M(Q, α) $ were introduced by Nakajima \cite{Nakajima} in 1994 and form another class of symplectic singularities. Among them are the Kleinian singularities, which are the most classical family of two-dimensional singularities:
\begin{center}
\begin{tikzpicture}
\path (0, 0) node[align=center] (A) {$ ℂ^2 \sslash Γ $ \\ \textbf{Kleinian singularities}};
\path (5, 0) node[align=center] (B) {$ \M(Q, α) $ \\ \textbf{Quiver varieties}};
\path (10, 0) node[align=center] (C) {$ (X, ω) $ \\ \textbf{Symplectic singularities}};
\path ($ (A.east)!0.5!(B.west) $) node {$ ⊂ $};
\path ($ (B.east)!0.5!(C.west) $) node {$ ⊂ $};
\end{tikzpicture}
\end{center}
Physicists are already vitally calculating Namikawa-Weyl groups of quiver varieties by means of heuristic methods and the hypothetical analogy with Coulomb branches. The aim of this paper is to put the description of the Namikawa-Weyl group on a rigorous footing. By Namikawa's semi-explicit procedure, we have to pick a symplectic resolution $ π: Y → X $ and then have three options:
{\renewcommand{\theenumi}{\alph{enumi}}
\renewcommand{\labelenumi}{(\theenumi)}
\begin{enumerate}
\item Construct universal deformations of $ Y $ and $ X $ and compare.
\item Determine the monodromy of $ π $ above every codimension-2 leaf.
\item Compare the deformation theory with variation of GIT.
\end{enumerate}}
\noindent The first approach was pursued successfully by Yaochen Wu \cite{Wu} in 2021. The second approach was pursued by the author in 2018 and is completed in the present paper. The third approach will soon be completed by Bellamy and Schedler. All three approaches yield the same result, but offer different potential for generalization to other symplectic singularities.

\addtocontents{toc}{\SkipTocEntry}
\subsection*{Strategy}
Our strategy for determining the Namikawa-Weyl groups is to explicitly identify the entire fiber bundle $ π: π^{-1} (L) → L $ for every codimension-2 leaf $ L $ of $ \M(Q, α) $. The implementation is best sketched as follows. For every point $ x = S_1 ⊕ … ⊕ S_k $ in $ L $, there is an astonishing similarity between the vertex representations $ \tilde S_1, …, \tilde S_k $ of the local quiver $ (Q', α') $ and the simples $ S_1, …, S_k $ of $ (Q, α) $. We bundle this similarity into an $ A_∞ $-functor $ F $ which allows us to transform representations in the exceptional locus of Kleinian singularities into representations lying in $ π^{-1} (L) $. This strategy provides an explicit identification of the bundle $ π: π^{-1} (L) → L $ and therefore an explicit description of the Namikawa-Weyl group of $ \M(Q, α) $.

In \autoref{sec:prelim}, we recapitulate preliminaries on symplectic singularities, quivers and $ A_∞ $-structures. In \autoref{sec:main-functor}, we construct the desired functor $ F $. In \autoref{sec:symmetry}, we investigate the symmetry group of the local quiver. In \autoref{sec:excfiber}, we identify the action of the symmetry group on the exceptional fibers of partial resolutions of Kleinian singularities. In \autoref{sec:monodromy}, we identify the fiber bundle explicitly and read off its monodromy. In \autoref{sec:main-result}, we collect the main results and provide an outlook on the the analogous situation with other Calabi-Yau-2 algebras. The appendices contain examples and several deferred proofs.

\addtocontents{toc}{\SkipTocEntry}
\subsection*{Acknowledgements}
This paper is a revamped version of the author's master's thesis prepared under supervision of Eric Opdam and Raf Bocklandt in 2017/18. The author is grateful to Travis Schedler for help with completing the paper in 2024, and to Imperial College London for hospitality.

%% file: prelim_short.tex
\section{Preliminaries}
\label{sec:prelim}
In this section, we recall preliminaries on quiver varieties, symplectic singularities, and $ A_∞ $-categories.

\addtocontents{toc}{\SkipTocEntry}
\subsection{Symplectic singularities}
We start by recalling the notion of symplectic singularities and their Namikawa-Weyl groups, following \cite{namikawa-II}. A \emph{symplectic singularity} is an irreducible normal variety, equipped with a non-degenerate closed algebraic 2-form $ ω $ on the smooth open locus $ X^{\reg} ⊂ X $, such that for any resolution $ π: Y → X $ with $ π: π^{-1} (X^{\reg}) → X^{\reg} $ bijective, the pullback form $ π^* ω $ extends to a regular algebraic 2-form on $ Y $.

Let $ X $ be an affine symplectic singularity equipped with an algebraic $ ℂ^× $-action with a unique fixed point $ 0 ∈ X $. The action has \emph{positive weights} if the weight spaces $ ℂ[X]_i \coloneqq \{f ∈ ℂ[X] \mid c{.}f = c^i f\} $ are all zero for $ i < 0 $ and we have $ ℂ[X]_0 = ℂ $. The symplectic form $ ω $ is \emph{weighted of degree $ l ∈ ℕ $} with respect to the $ ℂ^* $-action if under the automorphism $ φ_t: X → X $, $ x ↦ tx $ with $ t ∈ ℂ^* $, the form obeys $ φ_t^* ω = t^l ω $. The $ ℂ^* $-action on $ (X, ω) $ is \emph{good} if it has positive weights and $ ω $ is positively weighted.

It is useful to note that the Poisson structure on $ X^{\reg} $ extends to the whole of $ X $ by the assumption that $ X $ is normal \cite[Section 2.1]{kaledin-stratification}. This allows for the study of deformations of the Poisson structure, resulting in the following theorem of Namikawa:

\begin{theorem}[\cite{namikawa-II}]
Let $ (X, ω) $ be an affine symplectic variety with a good $ ℂ^× $-action. Let $ Y → X $ be a symplectic resolution. Then $ X $ and $ Y $ admit universal Poisson deformations $ \mathcal{X} → \HP^2 (X) $ and $ \mathcal{Y} → \HP^2 (Y) $. Moreover, we have $ \dim HP^2 (X) = \dim \HP^2 (Y) $ and there is a morphism of schemes connecting the universal Poisson deformations:
\begin{equation*}
\label{eq:prelim-namikawa-diagram}
\begin{tikzcd}
\mathcal{Y} \arrow[d] \arrow[r] & \mathcal{X} \arrow[d] \\
\HP^2 (Y) \arrow[r, "ψ"] & \HP^2 (X)
\end{tikzcd}
\end{equation*}
We have $ ψ(0) = 0 $ and the morphism $ \mathcal{Y} → \mathcal{X} $ restricts to the symplectic resolution $ Y → X $ on this fiber. The map $ ψ $ is a Galois covering between two sufficiently small punctured neighborhoods of zero. The Galois group does not depend on the choice of $ Y $.
\end{theorem}

The Galois covering group $ W $ is nowadays called the \emph{Namikawa-Weyl group}. This group can be described semi-explicitly as follows. By Theorem \cite[Theorem 2.3]{kaledin-stratification}, the symplectic variety $ X $ is stratified into finitely many smooth connected leaves, on which the Poisson structure becomes symplectic. This stratification is the algebraic analog of the stratification of a Poisson manifold in real-valued geometry. If $ x $ lies in a stratum $ B $, then locally analytically, $ X $ decomposes into a product of $ B $ and a “transversal” symplectic variety $ Z_x $.

Let $ B $ be a stratum in $ X $ of codimension $ 2 $. Then the transversal symplectic variety is of dimension $ 2 $ and hence locally a Kleinian singularity $ S $. This means that we have $ (X, x) \cong (B, x) × (S, 0) $. As the stratum is connected, the type of the transversal Kleinian singularity is equal on the entire stratum. Let $ n $ be the index of the Kleinian singularity. The symplectic resolution $ π: Y → X $ gives a germ of a minimal resolution $ π: π^{-1} (\{x\} × (S, 0)) → \{x\} × (S, 0) $ of the Kleinian singularity $ S $. This means that over each point $ x ∈ B $, there lie $ n + 1 $ projective line, intersecting according to the Dynkin diagram of $ S $. If we denote this shape by $ Δ $, the map $ π: π^{-1} (B) → B $ gives a $ Δ $-fiber bundle. Locally at a point $ x ∈ B $, in the $ D_4 $ case the fiber $ π^{-1} (B) $ looks like
\begin{equation*}
\begin{tikzpicture}
\draw (0, 0) circle (0.5) +(135:0.5) coordinate(x1) +(315:0.5) coordinate(x2);
\draw (1, 0) circle (0.5) +(135:0.5) coordinate(y1) +(315:0.5) coordinate(y2);
\node (a) at ({1 + 1/sqrt(2)}, {1/sqrt(2)}) {};
\node (b) at ({1 + 1/sqrt(2)}, {-1/sqrt(2)}) {};
\draw (a) circle (0.5) +(135:0.5) coordinate (c1) +(315:0.5) coordinate (c2);
\draw (b) circle (0.5) +(135:0.5) coordinate (d1) +(315:0.5) coordinate (d2);
\draw (c1) to[out=45, in=225] +(1, 1);
\draw (c2) to[out=45, in=225] +(1, 1);
\draw (d1) to[out=45, in=225] +(1, 1);
\draw (d2) to[out=45, in=225] +(1, 1);
\draw (x1) to[out=45, in=225] +(1, 1);
\draw (y1) to[out=45, in=225] +(1, 1);
\draw[fill=black] (0.5, 0) circle (0.06cm);
\draw[fill=black] ({1 + 0.5/sqrt(2)}, {0.5/sqrt(2)}) circle (0.06cm);
\draw[fill=black] ({1 + 0.5/sqrt(2)}, {-0.5/sqrt(2)}) circle (0.06cm);
\end{tikzpicture}
\end{equation*}
but globally on $ B $, some projective lines might be glued together, forming “tubes”.
\begin{equation*}
\begin{tikzpicture}
\draw (0, 0) circle (0.5);
\draw (1, 0) circle (0.5);
\node (a) at ({1 + 1/sqrt(2)}, {1/sqrt(2)}) {};
\node (b) at ({1 + 1/sqrt(2)}, {-1/sqrt(2)}) {};
\draw (a) circle (0.5) +(135:0.5) coordinate (c1) +(315:0.5) coordinate (c2);
\draw (b) circle (0.5) +(135:0.5) coordinate (d1) +(315:0.5) coordinate (d2);
\draw (c1) to[out=45, in=90] (4, 0) to[out=270, in=0] (1.5, -1.6) to[out=180, in=225] (d1);
\draw (c2) to[out=45, in=90] (3.6, 0) to[out=270, in=225] (d2);
\draw[fill=black] (0.5, 0) circle (0.06cm);
\draw[fill=black] ({1 + 0.5/sqrt(2)}, {0.5/sqrt(2)}) circle (0.06cm);
\draw[fill=black] ({1 + 0.5/sqrt(2)}, {-0.5/sqrt(2)}) circle (0.06cm);
\end{tikzpicture}
\end{equation*}
Let us consider the set of tubes, or more formally, the irreducible exceptional divisors. At $ x $, several projective lines may be connected by tubes. Because intersection points of projective lines are preserved along $ B $, there is a Dynkin graph automorphism $ φ $ that interchanges projective lines (nodes of the Dynkin diagram) that are glued together. Let $ W_B $ be the classical Weyl group associated to $ S $. The group $ W_B $ acts on the nodes of the Dynkin diagram by permutation. Define
\begin{equation*}
\tilde W_B \coloneqq \{w ∈ W \mid wφ = φw\},
\end{equation*}
the invariant part of $ W_B $ under the Dynkin automorphism. The following statement is Namikawa's semi-explicit description of the Galois group $ W $.

\begin{theorem}
\label{th:namikawa-weyl-explicit}
Let $ (X, ω) $ be an affine symplectic variety with a good $ ℂ^× $-action. Then, using the above notation, we have
\begin{equation*}
W = \prod_{B ∈ \mathcal{B}} \tilde W_B,
\end{equation*}
where $ \mathcal{B} $ denotes the set of symplectic codimension-2 strata (\cite[Theorem 1.1]{namikawa-II}).
\end{theorem}

Namikawa's proof shows that $ W $ is independent of the choice of $ Y $. Even better, the monodromy action of a loop $ γ $ in a codimension-2 stratum $ B $ can be read off purely from the transversal neighborhood of the stratum in $ X $, independent of $ Y $. The preimage of $ B $ in $ Y $ serves merely as a magnifying glass for tracing the monodromy action on the Kleinian singularity (namely, the monodromy action on the Kleinian singularity above a point is by symplectic automorphism, see Section 1.7 of \cite{namikawa-I}). This means that to compute the Dynkin automorphisms of all strata, we can choose a new symplectic resolution $ Y → X $ for each individual stratum. Note that $ W $ does not depend on the $ ℂ^× $-action either.

\addtocontents{toc}{\SkipTocEntry}
\subsection{Quiver varieties}
Next, let us recapitulate basic properies of quivers, quiver varieties and preprojective algebras. During the entire paper, we work with algebras and representations over the field of complex numbers. Let us start by recalling quivers and their representations.

A \emph{quiver} $ Q $ is a finite directed graph. We denote by $ Q_0 $ its set of vertices and by $ Q_1 $ its set of arrows. We denote by $ h: Q_1 → Q_0 $ and $ t: Q_1 → Q_0 $ the head and tail functions. A \emph{dimension vector} on $ Q $ is an element $ α ∈ ℕ^{Q_0} $. A \emph{quiver setting} $ (Q, α) $ is a quiver $ Q $ together with a dimension vector $ α ∈ ℕ^{Q_0} $.

A \emph{representation} or \emph{module} of $ Q $ of dimension $ α $ consists of assigning to every vertex $ i ∈ Q_0 $ a vector space $ V_i $ of dimension $ α_i $ and to every arrow $ a: i → j $ a linear mapping $ ρ(a): V_i → V_j $. The \emph{path algebra} $ ℂQ $ is the algebra with basis given by the paths in $ Q $, and the product of paths $ p · q $ being the concatenation of $ p $ and $ q $ if $ h(q) = t(p) $ and zero else.

The \emph{double quiver} $ \qQ $ of $ Q $ is obtained from $ Q $ by adding for every arrow $ a: i → j $ an arrow $ a^*: j → i $ in opposite direction. The \emph{preprojective algebra} $ Π_Q $ is defined as
\begin{equation*}
Π_Q = \frac{ℂ \qQ}{\big(\sum_{a ∈ Q_1} a a^* - a^* a\big)}.
\end{equation*}

The representations of a quiver can be bundled into a vector space. When enforcing the preprojective condition, we obtain an affine variety. We fix notation as follows:
\begin{equation*}
\Rep(\qQ, α) = \bigoplus_{\substack{a ∈ \qQ_1 \\ a: i → j}} \Hom(ℂ^{α_i}, ℂ^{α_j}).
\end{equation*}

The group $ \GL_α = \prod_{i ∈ Q_0} \GL_i (ℂ) $ acts on $ \Rep(\qQ, α) $ by conjugation:
\begin{equation*}
(g.ρ)_a = g_{h(a)} ρ(a) g_{t(a)}^{-1}.
\end{equation*}
The orbits of this action are the isomorphism classes of quiver representations of $ (\qQ, α) $. Recall that an affine variety $ X $ with an algebraic action of a reductive group $ G $ comes with an associated affine quotient $ X \quot G $, whose coordinate ring is the invariant ring $ ℂ[X]^G $.

This construction gives rise to quiver varieties: Let $ (Q, α) $ be a quiver setting. The \emph{moment map} $ μ: \Rep(\qQ, α) → (\Lie\GL_α)^* ≅ \Lie\GL_α $ of the double quiver is defined by
\begin{equation*}
μ(ρ) = \bigg( \sum_{h(a) = v} ρ(a)ρ(a^*) - \sum_{t(a) = v} ρ(a^*) ρ(a) \bigg)_{v ∈ Q_0}.
\end{equation*}
The \emph{quiver variety} $ \M(Q, α) $ is the affine quotient variety $ μ^{-1} (0) \quot \GL_α $. The points in $ \M(Q, α) $ are in one-to-one correspondence with the closed $ \GL_α $-orbits in $ μ^{-1} (0) $. In other words, these are the $ \GL_α $-orbits of semisimple representations in $ μ^{-1} (0) $.

The structure of the varieties $ \M(Q, α) $ has been investigated for more than 30 years, going back to work of Nakajima, Crawley-Boevey, King, and others. We shall base our inspection on the modern treatment presented by Bellamy and Schedler \cite{bellamy-schedler}.

\begin{theorem}[{\cite[Theorem 1.2]{bellamy-schedler}}]
Let $ (Q, α) $ be a quiver setting. Then $ \M(Q, α) $ is an irreducible symplectic singularity.
\end{theorem}

\addtocontents{toc}{\SkipTocEntry}
\subsection{Local structure of quiver varieties}
Quiver varieties have an interesting local structure. The systematic treatment requires a few combinatorial notions: A (positive) \emph{root} of $ Q $ is a dimension vector $ α $ of an indecomposable representation of $ Q $. A root $ α $ is \emph{real} if there is (up to isomorphism) a unique indecomposable representation with dimension vector $ α $. A root $ α $ is \emph{imaginary} if there are (up to isomorphism) multiple indecomposable representations with dimension vector $ α $. The set of roots is denoted $ R^+ $.

The \emph{Ringel form} $ ⟨-, -⟩ $ on $ ℂ^{Q_0} $ is given by
\begin{equation*}
⟨e_i, e_j⟩ = δ_{ij} - \#(a: i → j).
\end{equation*}
Its symmetrization is the \emph{Cartan form} $ (x, y) = ⟨x, y⟩ + ⟨y, x⟩ $ which can also be seen as $ (-, -) = 2 I_{Q_0} - A $, where $ I_{Q_0} $ denotes the identity and $ A $ denotes the adjacency matrix of the oriented graph $ \qQ $. We write $ p(α) = 1 - ½ (α, α) $.

Not all positive roots are created equal. Rather, there is a set $ Σ_{0, 0} $ of “good roots” defined by
\begin{multline*}
Σ_{0, 0} = \big\{α ∈ R^+ \running p(α) > \sum_{i = 1}^r p(β_i) \\
\text{for any nontrivial decomposition } α = β_1 + … + β_r, ~ (β_i) ⊂ R^+ \big\}.
\end{multline*}
An element $ α ∈ Σ_{0, 0} $ is \emph{minimal} if it cannot be written as a nontrivial sum of other elements of $ Σ_{0, 0} $.

There are quiver settings which yield empty quiver varieties. In fact, an explicit criterion states that the quiver variety is non-empty if and only if the dimension vector lies in the $ ℕ $-span of the positive roots:

\begin{theorem}[{\cite[Theorem 4.4]{cb-geometry-moment}}]
Let $ (Q, α) $ be a quiver setting. Then $ \M(Q, α) ≠ \emptyset $ implies $ α ∈ ℕR^+ $.
\end{theorem}

Thus, for nontrivial quiver varieties $ \M(Q, α) $ we can always assume $ α ∈ ℕR^+ $.

\begin{theorem}[{\cite[Proposition 2.1, Theorem 1.3]{bellamy-schedler}}]
Let $ Q $ be a quiver and $ α ∈ ℕR^+ $. Then $ α $ admits a decomposition $ α = n_1 α_1 + … + n_k α_k $ for some $ k ∈ ℕ $ and $ 0 < n_i ∈ ℕ $ and $ α_i ∈ Σ_{0, 0} $, such that any other decomposition with these properties is a refinement. This decomposition is called the \emph{canonical decomposition} of $ α $. In terms of the canonical decomposition, we have an isomorphism of symplectic varieties
\begin{equation*}
\M(Q, α) ≅ S^{n_1} \M(Q, α_k) × … × S^{n_k} \M(Q, α_k).
\end{equation*}
\end{theorem}

Thus, we can focus on the case that the dimension vector $ α $ itself lies in the set $ Σ_{0, 0} $. If $ α ∈ Σ_{0, 0} $, then $ α $ is its own canonical decomposition. Bellamy and Schedler also show that any dimension vector $ α $ lies in $ Σ_{0, 0} $ if and only if there is a simple representation of $ Q $ satisfying the preprojective condition \cite[Proposition 3.8]{bellamy-schedler}. If $ α ∈ Σ_{0, 0} $, then the variety $ \M(Q, α) $ has dimension $ 2 p(α) $.

Let us now recall the description of the symplectic stratification of $ \M(Q, α) $ due to Bellamy and Schedler. In particular, we recall the construction of local quivers associated with decompositions of $ α $, and the selected class of isotropic decompositions of $ α $.

\begin{theorem}[{\cite[Corollary 3.25]{bellamy-schedler}}]
Let $ (Q, α) $ be a quiver setting and $ α = n_1 β_1 + … + n_k β_k $ be a decomposition where $ n_i $ are positive integers and $ β_1, …, β_k $ are not necessarily distinct elements of $ Σ_{0, 0} $. Write $ τ = (n_1, β_1; …; n_k, β_k) $ and define $ \M(Q, α)_τ $ as the subset of $ \M(Q, α) $ consisting of all elements with semisimple lifts $ M ∈ \Rep(\qQ, α) $ that decompose into simples according to $ τ $:
\begin{equation*}
M ≅ \bigoplus_{i = 1}^k S_i^{⊕ n_i}, \quad \text{where } S_i ∈ \Rep^{\simp} (Q, β_i) \text{ with } S_i \not\cong S_j \text{ for } i ≠ j.
\end{equation*}
The $ S_i $ are supposed to be non-isomorphic 
Then the sets $ \M(Q, α)_τ $ form the symplectic stratification for $ \M(Q, α) $.
\end{theorem}

Let us recall the local structure of quiver varieties. Let $ (Q, α) $ be a quiver setting with $ α ∈ Σ_{0, 0} $ and let $ \M(Q, α)_τ $ be a symplectic stratum. Then the (reduced) \emph{local quiver} $ (\overline{Q'}, α') $ at $ τ $ is the quiver setting defined as follows:
\begin{itemize}
\item The vertices of $ \overline{Q'} $ are the integers $ 1, …, k $.
\item There are $ 2 p(β_i) $ loops at vertex $ i $.
\item There are $ -(β_i, β_j) $ arrows from vertex $ i $ to $ j $ when $ i ≠ j $.
\end{itemize}

Our notation of local quivers deviates slightly from the one used in \cite{bellamy-schedler}, where the full local quiver is denoted $ (Q', α') $ while the reduced local quiver is denoted $ (Q'', α'') $. The only difference between the two is that $ (Q'', α'') $ has all loops at vertices removed.

We are now ready to recall the local structure of $ \M(Q, α) $. It is possible to phrase the structure theory in étale local or formal scheme theory, but we shall work in the analytic setting for simplicity. We start with a point $ x ∈ \M(Q, α)_τ $. Then $ \M(Q, α)_τ $ is a smooth symplectic analytic manifold of dimension $ \sum_{i = 1}^k 2 p(β_i) $. The transversal slice to $ \M(Q, α)_τ $ at $ x $ is isomorphic to the germ of $ \M(Q', α') $ at zero:
\begin{equation*}
(\M(Q, α), x) ≅ (ℂ^{\sum 2 p(β_i)}, 0) × (\M(Q', α'), 0).
\end{equation*}
In particular, a stratum is of codimension 2 if and only if $ \M(Q', α') $ is two-dimensional. This is the case precisely when $ (Q', α') $ is Kleinian quiver setting. We can make a few further observations. For instance, every element $ β_i $ is either a real or an imaginary root. It is known that if $ β_i $ is a real root, then $ \M(Q, β_i) $ consists of a single point. Therefore if $ \M(Q, α)_τ $ is non-empty, we necessarily have that all real roots among $ β_1, …, β_k $ are pairwise distinct. To ensure that we enumerate only the non-empty codimension-2 leaves, we fix the following terminology, following the account of Bellamy and Schedler.

\begin{definition}
Let $ (Q, α) $ be a quiver setting with $ α ∈ Σ_{0, 0} $. Then a decomposition $ α = n_1 β_1 + … + n_k β_k $ is an \emph{isotropic decomposition} if
\begin{enumerate}
\item $ β_1, …, β_k ∈ Σ_{0, 0} $.
\item The real roots among $ β_1, …, β_k $ are pairwise distinct.
\item The local quiver $ (\overline{Q''}, α'') $ is a Kleinian quiver.
\end{enumerate}
\end{definition}

\addtocontents{toc}{\SkipTocEntry}
\subsection{Resolutions of quiver varieties}

In this section, we recall the result of Bellamy and Schedler on symplectic resolutions of quiver varieties. We start from a quiver variety $ \M(Q, α) $ where $ Q $ is a quiver and $ α $ does not necessarily in $ Σ_{0, 0} $. Bellamy and Schedler provide an explicit criterion for $ \M(Q, α) $ to have a symplectic resolution. We recall the criterion and provide basic insight into the resolution itself.

\begin{theorem}[\cite{bellamy-schedler}]
Let $ (Q, α) $ be a quiver setting and $ α = n_1 α_1 + … + n_k α_k $ be its canonical decomposition. Then $ \M(Q, α) $ admits a projective symplectic resolution if and only if every vector $ α_i $ is indivisible or satisfies $ (\gcd(α_i), p(\gcd(α_i)^{-1} α_i)) = (2, 2) $.
%
\end{theorem}

The theorem demands that every $ α_i $ either be indivisible or satisfy a combinatorial condition. We refer to these cases as the “indivisible case” and the “(2, 2) case”. Let us now assume that $ α ∈ Σ_{0, 0} $ and describe the symplectic resolution explicitly.

A symplectic resolution can often by found by installing a stability parameter. For quiver varieties $ \M(Q, α) $, this boils down to studying the Mumford quotient $ \M_θ (Q, α) = \Rep(Π_Q, α) \quot_θ \GL_α $. Here $ θ ∈ ℤ^{Q_0} $ is a stability parameter with $ θ · α = 0 $. The points of the quotient $ \M_θ (Q, α) $ are in one-to-one correspondence with $ θ $-polystable representations of $ Π_Q $ of dimension $ α $.

The variety $ \M_θ (Q, α) $ is sometimes smooth. The smoothness heavily depends on the choice stability parameter $ θ $. We fix the following terminology:

\begin{definition}
\label{def:prelim-generic}
Let $ (Q, α) $ be a quiver setting with $ α ∈ Σ_{0, 0} $. Let $ θ ∈ ℤ^{Q_0} $ be a stability parameter with $ θ · α = 0 $. Then
\begin{itemize}
\item $ θ $ is \emph{generic} if $ θ · β ≠ 0 $ for all $ 0 < β < α $.
\item If $ α $ is indivisible, then $ θ $ is \emph{pseudo-generic} if it is generic.
\item If $ (\gcd(α), p(\gcd(α)^{-1} α) = (2, 2) $, then $ θ $ is \emph{pseudo-generic} if $ θ · β ≠ 0 $ for all $ 0 < β < α $ with $ β ≠ α/2 $.
\end{itemize}
\end{definition}

The existence of generic stability parameters depends heavily on the dimension vector $ α $. In the indivisible case, the linear inequalities $ θ · β ≠ 0 $ cut out finitely many hyperplanes from the space orthogonal to $ α $ and almost every stability parameter $ θ $ with $ θ · α = 0 $ is generic. In the (2, 2) case, there is not a single generic stability parameter, but almost every stability parameter $ θ $ with $ θ · α = 0 $ is pseudo-generic.

We are now ready to recall the description of the symplectic resolution of $ \M(Q, α) $ in the indivisible and the (2, 2) case from the account of Bellamy and Schedler: 

\begin{theorem}[{\cite{bellamy-schedler}}]
Let $ (Q, α) $ be a quiver setting with $ α ∈ Σ_{0, 0} $. If $ α $ is indivisible, then $ \M_θ (Q, α) $ is a projective symplectic resolution for generic $ θ $. If $ (\gcd(α), p(\gcd(α)^{-1} α) = (2, 2) $, then $ \M_θ (Q, α) $ is a partial resolution for pseudo-generic $ θ $ and a symplectic resolution is obtained by blowing up $ \M_θ (Q, α) $ along the remaining singularity.
\end{theorem}

The remaining singularity of $ \M_θ (Q, α) $ in the (2, 2) case is quite interesting. We study its local structure in \autoref{sec:excfiber}.

\addtocontents{toc}{\SkipTocEntry}
\subsection{Calabi-Yau structure}
The module category $ \Mod_{Π_Q} $ has the structure of an $ A_∞ $-category and its precise structure is well-known. To start with, recall from e.g.~\cite{Bocklandt-book} that an \emph{$ A_∞ $-category} $ \cat C $ consists of a set of $ ℤ $-graded hom spaces together with higher products $ μ^k $ for $ k ≥ 1 $ of degree $ 2 - k $ satisfying the $ A_∞ $-relations
\begin{equation*}
\sum_{k ≥ l ≥ m ≥ 1} (-1)^{‖a_m‖ + … + ‖a_1‖} μ(a_k, …, μ(a_l, …, a_{m+1}), …, a_1) = 0.
\end{equation*}
The category $ \cat C $ is \emph{cyclic} of degree $ n $ if it is equipped with a bilinear nondegenerate pairing $ ⟨-, -⟩: \Hom(X, Y) × \Hom(Y, X) → ℂ $ of degree $ -n $ such that $ ⟨x, y⟩ = (-1)^{|x||y|} ⟨y, x⟩ $ and
\begin{equation*}
⟨μ^k (a_{k+1}, …, a_2), a_1⟩ = (-1)^{‖a_{k+1}‖ (‖a_k‖ + … + ‖a_1‖)} ⟨μ^k (a_k, …, a_1), a_{k+1}⟩.
\end{equation*}
We can now regard the specific $ A_∞ $-category $ \Mod_{Π_Q} $. The objects of this category are the finite-dimensional $ Π_Q $-modules. The hom spaces are $ ℤ $-graded and given by $ \Hom(M, N) = \Ext^{•} (M, N) $. The differential $ μ^1 $ vanishes, and the product $ μ^2 $ is the signed (Yoneda) product of morphisms and extensions. The algebra $ Π_Q $ has a Calabi-Yau property. The exception of Dynkin quivers here refers to the non-extended (finite type) Dynkin quivers.

\begin{theorem}[\cite{CrawleyBoevey-Kimura}]
If $ Q $ is not a Dynkin quiver, then the algebra $ Π_Q $ is Calabi-Yau-2.
\end{theorem}

The CY-2 property of $ Π_Q $ induces a Calabi-Yau structure on the category $ \Mod_{Π_Q} $. Given that the category $ \Mod_{Π_Q} $ arises from a minimal model construction and is only defined up to isomorphism, it is tricky to make its Calabi-Yau structure explicit. However, there exists a model of $ \Mod_{Π_Q} $ whose $ A_∞ $-structure is cyclic. We are not aware of a specific proof of the cyclicity in the literature, but it appears to be well-known among experts and we shall simply assume it:

\begin{theorem}
\label{th:prelim-cyclicity}
If $ Q $ is not a Dynkin quiver, then the $ A_∞ $-category $ \Mod_{Π_Q} $ is cyclic.
\end{theorem}

Our next step is to describe the $ A_∞ $-structure of $ \Mod_{Π_Q} $. We shall at this point not aim to describe the entire structure, but focus on a collection of pairwise non-isomorphic simples $ S_1, …, S_k $ associated with a point in a codimension-2 leaf. Vanishing of higher products in an $ A_∞ $-category is also known as formality. Fortunately, Davison has already investigated this question and found that any $ Σ $-sequence in a Calabi-Yau-2 algebra is formal. While this formality theorem in principle only asserts that there exists a model with vanishing higher products, we will simply assume that this model agrees with the cyclic model:

\begin{theorem}[{\cite[Theorem 1.2]{Davison}}]
Let $ Q $ be a non-Dynkin quiver and let $ S_1, …, S_k $ be simple finite-dimensional representations of $ Π_Q $. Then the higher products $ μ^{≥3} $ on $ \Mod_{Π_Q} $ vanish.
\end{theorem}

The dimensions of ext spaces between non-isomorphic representations can be determined in a combinatorial way, which we record as follows:

\begin{theorem}[{\cite[Lemma 1]{cb-kleinian}}]
\label{th:prelim-extdim}
If $ M $ and $ N $ are finite-dimensional representations of $ Π_Q $, then we have
\begin{equation*}
\dim \Ext^1 (M, N)  = \dim \Hom (M, N) + \dim \Hom(N, M) - (\dim M, \dim N).
\end{equation*}
\end{theorem}

%% file: result_short.tex
\section{Local-to-global functor}
\label{sec:main-functor}
Thanks to the description of the $ A_∞ $-structure of $ \Mod_{Π_Q} $, we can immediately construct a correspondence between the local and the global representations. More precisely, we can construct an $ A_∞ $-functor $ F: \{\tilde S_1, …, \tilde S_k\} → \{S_1, …, S_k\} $ which is a strict faithful unital embedding. Indeed, we simply map $ \tilde S_i $ to $ S_i $, map identity elements to identity elements, co-identity elements to co-identity elements and choose arbitrary linear graded identifications between the hom spaces $ \Hom_{Π_{Q'}} (\tilde S_i, \tilde S_j) $ and $ \Hom_{Π_Q} (S_i, S_j) $ for $ i ≠ j $.

The functor can be extended to the twisted completions of its domain and codomain categories. It then becomes interesting to study those twisted complexes which are quasi-isomorphic in $ \Tw\Mod_{Π_Q} $ to actual modules. We call such twisted objects \emph{honest modules}. The functor $ F $ behaves nicely with respect to honest modules, and we collect the important statements in \autoref{th:result-functor}. The proof is provided for reference in \autoref{sec:functor}.

\begin{proposition}
\label{th:result-functor}
Let $ Q $ be a non-Dynkin quiver. Let $ S_1, …, S_k $ be simples $ S_i ∈ \Mod_{Π_Q} $. Denote by $ Q' $ the local quiver and let $ \tilde S_1, …, \tilde S_k ∈ \Mod_{Π_{Q'}} $ be its standard simples. Then there is a unital strict faithful $ A_∞ $-embedding
\begin{align*}
F: \{\tilde S_1, …, \tilde S_k\} &\verylongisoto \{S_1, …, S_k\}, \\
\tilde S_i &\verylongmapsto S_i.
\end{align*}
The induced functor $ \H F: \HTw\{\tilde S_1, …, \tilde S_k\} → \HTw\{S_1, …, S_k\} $ is an embedding as well and sends honest modules to honest modules. Moreover, if the images of two honest modules are isomorphic, then the honest modules are already isomorphic.
\end{proposition}

Every stability parameter $ θ ∈ ℤ^{Q_0} $ of $ Q $ gives rise to the \emph{localized stability parameter} $ θ' ∈ ℤ^{Q'_0} $ of $ Q' $ given by
\begin{equation*}
θ' = (β_1 · θ, …, β_k · θ) ∈ ℤ^{Q'_0}.
\end{equation*}

The functor $ F $ has very beneficial properties with respect to stable, semistable and polystable objects, which we state in \autoref{th:main-functor-stability}. The proof is easy and provided for reference in \autoref{sec:functor-stability}.

\begin{lemma}
\label{th:main-functor-stability}
Let $ (Q, α) $ be a quiver setting and $ S_1, …, S_k ∈ \Mod_{Π_Q} $ be simple representations. Let $ Q' $ be the local quiver and $ \tilde S_1, …, \tilde S_k $ its standard simples. Let $ θ ∈ ℤ^{Q_0} $ be a stability parameter and $ θ' $ its localized version. Let $ X ∈ \Tw\{\tilde S_1, …, \tilde S_k\} $ be an honest module. Then
\begin{equation*}
X \text{ is } θ' \text{-semistable/stable/polystable} \quad \Longleftrightarrow \quad F(X) \text{ is } θ \text{-semistable/stable/polystable}.
\end{equation*}
\end{lemma}

\section{Symmetries of the labeled local quiver}
\label{sec:symmetry}
Next, let us investigate symmetries among the roots $ β_1, …, β_k $ associated with a codimension-2 leaf. As we see in \autoref{th:result-symmetry}, every symmetry among the roots $ β_1, …, β_k $ actually comes from a symmetry of the whole labeled graph $ (Q', α', (β_1, …, β_k)) $. The proof is easily done by case-checking, using the information provided by the pairings $ (β_i, β_j) $ and is elaborated in \autoref{sec:app-symmetry}.

\begin{proposition}
\label{th:result-symmetry}
Let $ (Q, α) $ be a quiver setting with $ α ∈ Σ_{0, 0} $. Let $ L $ be a codimension-2 leaf, given by the roots $ β_1, …, β_k $. Let $ σ ∈ S_k $ be a permutation which preserves the roots $ β_1, …, β_k $. Then $ σ $ is trivial or a composition of one of the primitive symmetries listed in \autoref{fig:symmetry-symmetry-possible}. In particular, we have $ σ ∈ \Aut(Q', α', (β_1, …, β_k)) $.
\end{proposition}

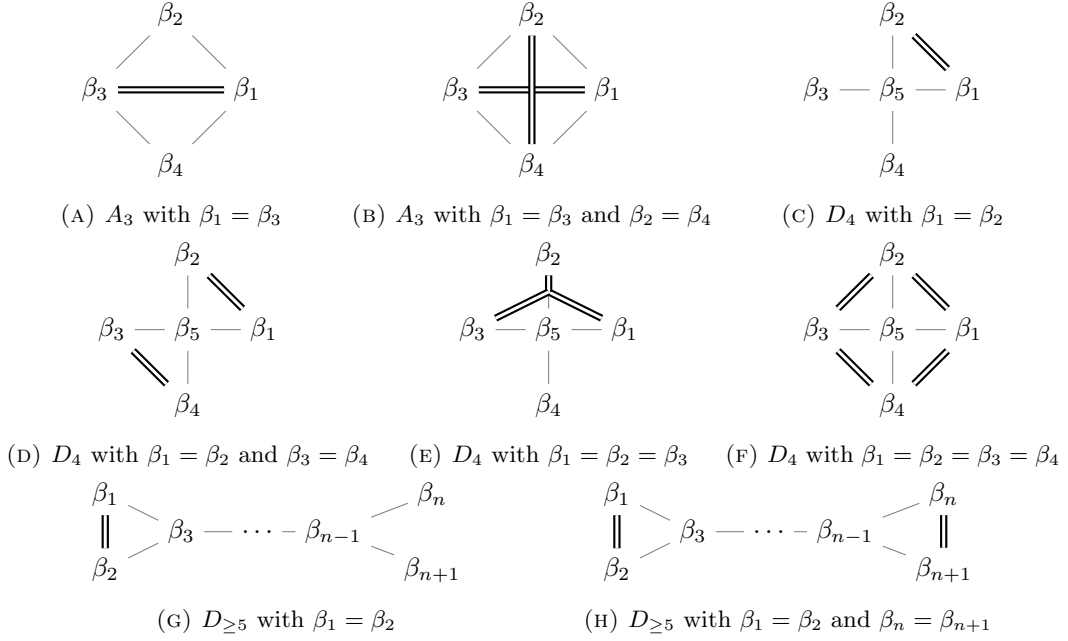
\begin{figure}
\centering
\begin{subfigure}{0.3\linewidth}
\centering
\begin{tikzpicture}
\path (0:1) node (A) {$ β_1 $};
\path (90:1) node (B) {$ β_2 $};
\path (180:1) node (C) {$ β_3 $};
\path (270:1) node (D) {$ β_4 $};
\path[draw, gray, -] (A) to (B);
\path[draw, gray, -] (B) to (C);
\path[draw, gray, -] (C) to (D);
\path[draw, gray, -] (D) to (A);
\path[draw, thick, double equal sign distance] (C) to (A);
\end{tikzpicture}
\caption{$ A_3 $ with $ β_1 = β_3 $}
\label{fig:symmetry-symmetry-possible-A3-1}
\end{subfigure}
\begin{subfigure}{0.33\linewidth}
\centering
\begin{tikzpicture}
\path (0:1) node (A) {$ β_1 $};
\path (90:1) node (B) {$ β_2 $};
\path (180:1) node (C) {$ β_3 $};
\path (270:1) node (D) {$ β_4 $};
\path[draw, gray, -] (A) to (B);
\path[draw, gray, -] (B) to (C);
\path[draw, gray, -] (C) to (D);
\path[draw, gray, -] (D) to (A);
\path[draw, thick, double equal sign distance] (C) to (A);
\path[draw, thick, double equal sign distance] (B) to (D);
\end{tikzpicture}
\caption{$ A_3 $ with $ β_1 = β_3 $ and $ β_2 = β_4 $}
\label{fig:symmetry-symmetry-possible-A3-2}
\end{subfigure}
\begin{subfigure}{0.3\linewidth}
\centering
\begin{tikzpicture}
\path (0:1) node (A) {$ β_1 $};
\path (90:1) node (B) {$ β_2 $};
\path (180:1) node (C) {$ β_3 $};
\path (270:1) node (D) {$ β_4 $};
\path (0, 0) node (E) {$ β_5 $};
\path[draw, gray, -] (A) to (E);
\path[draw, gray, -] (B) to (E);
\path[draw, gray, -] (C) to (E);
\path[draw, gray, -] (D) to (E);
\path[draw, thick, double equal sign distance] (A) to (B);
\end{tikzpicture}
\caption{$ D_4 $ with $ β_1 = β_2 $}
\end{subfigure}
\begin{subfigure}{0.33\linewidth}
\centering
\begin{tikzpicture}
\path (0:1) node (A) {$ β_1 $};
\path (90:1) node (B) {$ β_2 $};
\path (180:1) node (C) {$ β_3 $};
\path (270:1) node (D) {$ β_4 $};
\path (0, 0) node (E) {$ β_5 $};
\path[draw, gray, -] (A) to (E);
\path[draw, gray, -] (B) to (E);
\path[draw, gray, -] (C) to (E);
\path[draw, gray, -] (D) to (E);
\path[draw, thick, double equal sign distance] (A) to (B);
\path[draw, thick, double equal sign distance] (C) to (D);
\end{tikzpicture}
\caption{$ D_4 $ with $ β_1 = β_2 $ and $ β_3 = β_4 $}
\end{subfigure}
\begin{subfigure}{0.3\linewidth}
\centering
\begin{tikzpicture}
\path (0:1) node (A) {$ β_1 $};
\path (90:1) node (B) {$ β_2 $};
\path (180:1) node (C) {$ β_3 $};
\path (270:1) node (D) {$ β_4 $};
\path (0, 0) node (E) {$ β_5 $};
\path[draw, gray, -] (A) to (E);
\path[draw, gray, -] (B) to (E);
\path[draw, gray, -] (C) to (E);
\path[draw, gray, -] (D) to (E);
\path[draw, thick, double equal sign distance] (A) to (0, 0.5);
\path[draw, thick, double equal sign distance] (B) to (0, 0.5);
\path[draw, thick, double equal sign distance] (C) to (0, 0.5);
\end{tikzpicture}
\caption{$ D_4 $ with $ β_1 = β_2 = β_3 $}
\end{subfigure}
\begin{subfigure}{0.3\linewidth}
\centering
\begin{tikzpicture}
\path (0:1) node (A) {$ β_1 $};
\path (90:1) node (B) {$ β_2 $};
\path (180:1) node (C) {$ β_3 $};
\path (270:1) node (D) {$ β_4 $};
\path (0, 0) node (E) {$ β_5 $};
\path[draw, gray, -] (A) to (E);
\path[draw, gray, -] (B) to (E);
\path[draw, gray, -] (C) to (E);
\path[draw, gray, -] (D) to (E);
\path[draw, thick, double equal sign distance] (A) to (B);
\path[draw, thick, double equal sign distance] (B) to (C);
\path[draw, thick, double equal sign distance] (C) to (D);
\path[draw, thick, double equal sign distance] (D) to (A);
\end{tikzpicture}
\caption{$ D_4 $ with $ β_1 = β_2 = β_3 = β_4 $}
\end{subfigure}
\begin{subfigure}{0.45\linewidth}
\centering
\begin{tikzpicture}
\path (0, 0.5) node (B) {$ β_1 $};
\path (0, -0.5) node (A) {$ β_2 $};
\path (1, 0) node (C) {$ β_3 $};
\path (2, 0) node (D) {$ … $};
\path (3, 0) node (E) {$ β_{n-1} $};
\path (4.3, 0.5) node (F) {$ β_n $};
\path (4.3, -0.5) node (G) {$ β_{n+1} $};
\path[draw, gray, -] (A) to (C);
\path[draw, gray, -] (B) to (C);
\path[draw, gray, -] (C) to (D);
\path[draw, gray, -] (D) to (E);
\path[draw, gray, -] (E) to (F);
\path[draw, gray, -] (E) to (G);
\path[draw, thick, double equal sign distance] (A) to (B);
\end{tikzpicture}
\caption{$ D_{≥5} $ with $ β_1 = β_2 $}
\end{subfigure}
\begin{subfigure}{0.45\linewidth}
\centering
\begin{tikzpicture}
\path (0, 0.5) node (B) {$ β_1 $};
\path (0, -0.5) node (A) {$ β_2 $};
\path (1, 0) node (C) {$ β_3 $};
\path (2, 0) node (D) {$ … $};
\path (3, 0) node (E) {$ β_{n-1} $};
\path (4.3, 0.5) node (F) {$ β_n $};
\path (4.3, -0.5) node (G) {$ β_{n+1} $};
\path[draw, gray, -] (A) to (C);
\path[draw, gray, -] (B) to (C);
\path[draw, gray, -] (C) to (D);
\path[draw, gray, -] (D) to (E);
\path[draw, gray, -] (E) to (F);
\path[draw, gray, -] (E) to (G);
\path[draw, thick, double equal sign distance] (A) to (B);
\path[draw, thick, double equal sign distance] (F) to (G);
\end{tikzpicture}
\caption{$ D_{≥5} $ with $ β_1 = β_2 $ and $ β_n = β_{n+1} $}
\end{subfigure}
\caption{This figure depicts the possible symmetries among roots in the local quiver. Long equality signs in the form of double strokes have been used to indicate equality among roots. For better readability, the arrows of the quiver have been grayed out. For brevity, we have depicted similar situations only once. For instance, in the $ D_4 $ case it is equally possible that $ β_2 = β_3 $ instead of $ β_1 = β_2 $.}
\label{fig:symmetry-symmetry-possible}
\end{figure}

\section{Action on the Kleinian exceptional fiber}
\label{sec:excfiber}
Next, we show that a symmetry in the local quiver induces monodromy of the resolution. Regard the local quiver $ (Q', α') $ associated with the leaf $ L $. Let $ σ $ be an automorphism of $ (Q', α') $ which keeps the roots $ β_i $ invariant. Let us denote by $ \tilde{φ} $ the map $ \tilde{φ}: \Rep(Π_{Q'}, α') → \Rep(Π_{Q'}, α') $ which swaps a given representation around according to $ σ $, inserting $ ±1 $ signs as necessary. The map $ \tilde{φ}: \Rep(Π_{Q'}, α') → \Rep(Π_{Q'}, α') $ descends to the GIT quotient and to the Mumford by any stability parameter. This is automatic by the construction of GIT and we obtain the commutative diagram
\begin{center}
\begin{tikzcd}
\Rep(π_{Q'}, α') \arrow[r, "\tilde{φ}"] \arrow[d] & \Rep(π_{Q'}, α') \arrow[d] \\
\M_{θ'} (Q', α') \arrow[r, "φ_{θ'}"] \arrow[d] & \M_{θ'} (Q', α') \arrow[d] \\
\M (Q', α') \arrow[r, "φ"] & \M (Q', α')
\end{tikzcd}
\end{center}
Using this diagram, we observe that the map $ φ_{θ'} $ acts by actually flipping a representation according to $ σ $, even on the exceptional fiber.

\begin{proposition}
\label{th:main-excfiber-action}
Let $ (Q, α) $ be a quiver setting with $ α ∈ Σ_{0, 0} $. Assume $ α $ is indivisible or $ (\gcd(α), p(\gcd(α)^{-1} α)) = (2, 2) $. Let $ θ ∈ ℤ^{Q_0} $ be a pseudo-generic stability parameter. Let $ L $ be a codimension-2 leaf with local quiver $ (Q', α') $ and localized stability parameter $ θ' ∈ ℤ^{Q'_0} $. Let $ σ ∈ \Aut(Q', α', (β_1, …, β_k)) $ be a primitive automorphism, one from the collection depicted in \autoref{fig:symmetry-symmetry-possible}. Then $ σ $ acts on the exceptional fiber of $ \M_{θ'} (Q', α') $ by Dynkin automorphism of the same degree.

\end{proposition}

\begin{proof}
We provide here only a very short sketch of the proof. The rest can be found in the appendix. We first deal with the case that $ σ $ is trivial and $ θ' $ is the specific stability parameter $ θ = (-K, +1, …, +1) $, as chosen by Crawley-Boevey. Then $ \M_{θ'} (Q', α') $ is smooth and by \cite{cb-kleinian} the $ n+1 $ many projective lines in the special fiber are distinguished by their socle, which is one of the vertex simples. As we let $ σ $ act on a representation, the socle permutes accordingly and we conclude the desired statement.

If $ θ' $ is any generic stability parameter, the statement follows from the diagram comparing the two resolutions. The case where $ (Q', α') $ is $ A_3 $ and $ σ $ swaps two pairs of vertices, or where $ (Q', α') $ is $ D_4 $ and $ σ $ involves all outer vertices, is performed in \autoref{sec:app-excfiber-all}. This exhausts all options.
\end{proof}

\section{Identification of monodromy}
\label{sec:monodromy}
We are now ready to describe the monodromy of the bundle $ π: π^{-1} (L) → L $. We start with a quiver setting $ (Q, α) $ with $ α ∈ Σ_{0, 0} $ such that $ α $ is indivisible or $ (\gcd(α), p(\gcd(α)^{-1} α)) = (2, 2) $. By Namikawa's description, we are supposed to pick a symplectic resolution $ π: Y → \M(Q, α) $ and read off the monodromy of the fiber bundle $ π: π^{-1} (L) → L $ over every codimension-2 leaf $ L $. It is our aim to analyze the monodromy of the bundle $ π: π^{-1} (L) → L $, where $ π $ is a symplectic resolution of $ \M(Q, α) $. If $ α $ is indivisible, then a symplectic resolution is given by $ π: \M_θ (Q, α) → \M(Q, α) $ where $ θ $ is a generic stability parameter. If $ (\gcd(α), p(\gcd(α)^{-1} α)) = (2, 2) $, then a symplectic resolution is given by the blowup of $ \M_θ (Q, α) $ along the remaining singularity, where $ θ $ is a pseudo-generic stability parameter.

If $ (\gcd(α), p(\gcd(α)^{-1} α)) = (2, 2) $, then the partial resolution $ \M_θ (Q, α) $ with $ θ $ pseudo-generic is already smooth enough to read off the desired monodromy. Indeed, the special locus is examined in \autoref{sec:app-monodromy} and it suffices to trace the two or three remaining singular points around the leaf.

The second core ingredient in our identification of the monodromy is the functor $ F $ from \autoref{sec:main-functor}. We start with a codimension-2 leaf $ L $, given by the isotropic decomposition $ α = m_1 β_1 + … + m_k β_k $.
Let $ S_1, …, S_k ∈ \Mod_{Π_Q} $ be a collection of simples of dimensions $ β_1, …, β_k $ and let $ \tilde S_1, …, \tilde S_k ∈ \Mod_{Π_{Q'}} $ be the vertex simples of the Kleinian local quiver $ (Q', α') $. Then we have the functor
\begin{equation*}
F_{S_1, …, S_k}: \Tw\{\tilde S_1, …, \tilde S_k\} \to \Tw\{S_1, …, S_k\}.
\end{equation*}
We bundle these functors $ F $ together into a single function whose input is a collection $ (S_1, …, S_k) $. We write
\begin{equation*}
M ≔ \big(\M^{\simp} (Q, β_1) × … × \M^{\simp} (Q, β_k)\big) \setminus Δ.
\end{equation*}
Here $ Δ ⊂ \prod_{i = 1}^k \M^{\simp} (Q, β_i) $ denotes the subset of tuples where at least two representations are isomorphic. Note that the space $ M $ is not the same as the leaf $ L $, but rather $ L = M / Γ $ where $ Γ = \Aut(β_1, …, β_k) $ acts naturally on $ M $.

Let now $ θ ∈ ℤ^{Q_0} $ be a stability parameter and $ θ' ∈ ℤ^{Q'_0} $ its localized version. Denote by $ E $ the exceptional fiber of $ π: \M_{θ'} (Q', α') \to \M (Q', α') $. Then we can interpret the family of functors as a continuous function
\begin{equation*}
F: M \times E \to \M_θ (Q, α).
\end{equation*}
The group $ Γ $ acts on $ M $ and also on $ E $ by the lift of the action on $ \M(Q, α) $. The map $ F $ is in fact $ Γ $-invariant, analogous to the fact that “rotating the building plan of a house and simultaneously permuting the bricks gives the same house”, see \autoref{rem:app-monodromy-invariant}. We conclude that $ F $ descends to a map
\begin{equation*}
F: (M × E) / Γ → \M_θ (Q, α).
\end{equation*}

We immediately obtain the following result:

\begin{proposition}
\label{th:monodromy-fiberbundle}
Let $ (Q, α) $ be a quiver setting with $ α ∈ Σ_{0, 0} $. Let $ L $ be a codimension-2 leaf with local quiver $ (Q', α') $ and roots $ β_1, …, β_k $. Let $ θ ∈ ℤ^{Q_0} $ be a pseudo-generic stability parameter and let $ θ' ∈ ℤ^{Q'_0} $ be the localized stability parameter. Let $ E $ be the exceptional fiber of $ \M_{θ'} (Q', α') $. Write $ Γ = \Aut(Q', α', (β_1, …, β_k)) $. Then
\begin{enumerate}
\item The map $ π_M: (M × E) / Γ → M / Γ $ is a fiber bundle.
\item The bundles $ π $ and $ π_M $ are isomorphic via $ F $:
\begin{center}
\begin{tikzcd}
(M × E) / Γ \arrow[d, "π_M"] \arrow{r}{F}[swap]{\sim} & π^{-1} (L) \arrow[d, "π"] \\
M / Γ \arrow[r, "="] & L
\end{tikzcd}
\end{center}
\item The monodromy action of $ π_1 (L) $ on $ E $ is given by the composition $ π_1 (L) → Γ → \Aut(E) $.
\end{enumerate}
\end{proposition}

We can reformulate the proposition more explicitly. Start by choosing an arbitrary pseudo-generic stability parameter $ θ ∈ ℤ^{Q_0} $, which necessarily exists due to the assumptions on $ α $. Then apply the explicit description of the action of $ Γ $ on the exceptional fiber $ E ⊂ \M_{θ'} (Q', α') $ from \autoref{th:main-excfiber-action}. We obtain the following explicit explicit description of the Namikawa-Weyl group:

\begin{theorem}
\label{th:monodromy-result-th}
Let $ (Q, α) $ be a quiver setting with $ α ∈ Σ_{0, 0} $. Let $ L $ be a codimension-2 leaf with local quiver $ (Q', α') $ and roots $ β_1, …, β_k $. Then the following provides an exhaustive list of the possible equalities among the roots $ β_1, …, β_k $ and the Weyl groups associated with the leaf $ L $.

In case $ L $ is of $ A_n $ type, we have the following distinction:
\begin{enumerate}
\item If $ n = 1, 2, ≥ 4 $, then $ W_L $ is $ A_n $.
\item If $ n = 3 $ and all roots are distinct, then $ W_L $ is $ A_3 $.
\item If $ n = 3 $ and one pair of opposite roots agrees, then $ W_L $ is $ C_2 $.
\item If both pairs of opposite roots agree, then $ W_L $ is $ C_2 $.
\end{enumerate}
In case $ L $ is of $ D_n $ type, we have the following distinction:
\begin{enumerate}
\item If all roots are distinct, then $ W_L $ is $ D_n $.
\item If $ n = 4 $ and one pair of roots agrees, then $ W_L $ is $ B_3 $.
\item If $ n = 4 $ and one triple of roots agrees, then $ W_L $ is $ G_2 $.
\item If $ n = 4 $ and one quadruple of roots agrees, then $ W_L $ is $ G_2 $.
\item If $ n = 4 $ and two pairs of roots agree, then $ W_L $ is $ B_3 $.
\item If $ n ≥ 5 $ and one pair of roots agrees, then $ W_L $ is of $ B_{n-1} $.
\item If $ n ≥ 5 $ and two pairs of roots agree, then $ W_L $ is of $ B_{n-1} $.
\end{enumerate}
In case $ L $ is of $ E_6 $, $ E_7 $ or $ E_8 $ type, then $ W_L $ is $ E_6 $, $ E_7 $, $ E_8 $, respectively.
\end{theorem}

\section{Main result and further applications}
\label{sec:main-result}
Let us now generalize to the case $ α \notin Σ_{0, 0} $. We use the following simple lemma for the Namikawa-Weyl groups $ W $ of products and symmetric products of symplectic singularities, which is proven in \autoref{sec:app-products}.

\begin{lemma}
Let $ X, X_1, …, X_k $ be symplectic singularities with good $ ℂ^* $-action. Then $ W(X_1 × … × W_k) = W(X_1) × … × W(X_k) $ and $ W(S^n X) = W(X) $.
\end{lemma}

\begin{theorem}
Let $ (Q, α) $ be a quiver setting with canonical decomposition $ α = n_1 α_1 + … + n_k α_k $. Assume that for every $ i = 1, …, k $, the dimension vector $ α_i $ is indivisible or satisfies $ (\gcd(α), p(\gcd(α)^{-1} α)) = (2, 2) $. Then the Namikawa-Weyl group of $ \M(Q, α) $ is given by
\begin{equation*}
W(\M(Q, α)) = W(\M(Q, α_1)) × … × W(\M(Q, α_k)).
\end{equation*}
Here $ W(\M(Q, α_i)) $ denotes the Namikawa-Weyl group of $ \M(Q, α_i) $, which can be determined according to \autoref{th:monodromy-result-th}.
\end{theorem}

While this paper focuses on quiver varieties, the same strategy extends naturally to moduli spaces of representations of Calabi-Yau-2 algebras. Specifically, for a Calabi-Yau-2 algebra $ Π $ and a dimension vector $ α $, consider a moduli space $ \M(Π, α) $ of $ α $-dimensional representations. If $ \M(Π, α) $ is a symplectic singularity, the codimension-2 strata take the form
\begin{equation*}
L = (\M^o (Π, β_1)_{c_1} × … × \M^o (Π, β_k)_{c_k}) / Γ.
\end{equation*}
Here $ β_1, …, β_k $ are the local dimension vectors. The letters $ c_1, …, c_k $ denote additional dimensional data which serve to specify individual connected components of $ \M^o (Π, β_k) $. The letter $ Γ $ denotes the permutation group which fixes the dimensional inputs $ (β_i, c_i) $. The monodromy of the stratum $ L $ is then equal to the action of $ Γ $ on the Kleinian exceptional fiber. This strategy determines the Namikawa-Weyl group for $ \M(Π, α) $.

%% file: app_examples/intro.tex
\section{Examples}
In this section, we present several examples of quiver settings $ (Q, α) $ together with their associated Namikawa-Weyl group. In these examples, we aim to illustrate how the Namikawa-Weyl group can be deduced from Namikawa's semi-explicit description, how to immediately read off the Namikawa-Weyl group from our main result and why both agree.

%% file: app_examples/example.tex
\addtocontents{toc}{\SkipTocEntry}
\subsection{An example}
In this section, we study an example quiver variety $ X = \M(Q, α) $ and explain how to find its Namikawa-Weyl group. We follow the semi-explicit description laid out by Namikawa for finding the Namikawa-Weyl group of $ X $ via symplectic resolutions:
\begin{enumerate}
\item Find the symplectic codimension-2 leaves $ L_1, …, L_k $ of $ X $.
\item Construct a symplectic resolution $ π: Y → X $.
\item Identify the monodromy of the fiber bundles $ π: π^{-1} (L_i) → L_i $.
\item Define $ W_i $ as the subgroup of the classical Weyl group invariant under the monodromy.
\item The Namikawa-Weyl group is then $ W = W_1 × … × W_k $.
\end{enumerate}
For the specific example $ X = \M(Q, α) $, we perform these five steps as follows: First, we identify the symplectic codimension-2 leaves by means of isotropic decompositions of $ α $. Second, we provide an explicit symplectic resolution $ π: \M_θ (Q, α) → \M(Q, α) $ by choosing a stability parameter $ θ $. Third, we analyse the representations in the exceptional fiber $ π^{-1} (L) $ over every codimension-2 leaf $ L $. The fourth and the fifth step are automatic and we find the Namikawa-Weyl group of $ \M(Q, α) $.

The starting point for our example is the double quiver $ \qQ $ with dimension vector $ α = (1, 2, 1) $ depicted in \autoref{fig:intro-example-quiver}. We label the three vertices as 1, 2 and 3, as depicted in the figure. There is no canonical choice of quiver $ Q $ which produces this double quiver, but for any choice of $ Q $ we see that $ Q $ has an indecomposable representation of dimension $ α $. This means that $ α $ lies in the set $ Σ_{0, 0} $ of “good dimension vectors”. This makes it easy to identify the symplectic codimension-2 leaves according to the procedure of Bellamy and Schedler \cite{bellamy-schedler}.

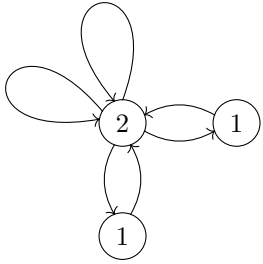
\begin{figure}
\centering
\begin{tikzpicture}
\path (0, 0) node[circle, draw] (A) {1} (0, 1.5) node[circle, draw] (B) {2} (1.5, 1.5) node[circle, draw] (C) {1};
\path[draw, ->, bend right] ($ (A) + (70:0.3) $) to node[midway, right] {} ($ (B) + (290:0.3) $);
\path[draw, <-, bend left] ($ (A) + (110:0.3) $) to node[midway, left] {}  ($ (B) + (250:0.3) $);
\path[draw, ->, bend right] ($ (B) + (340:0.3) $) to node[midway, below] {} ($ (C) + (200:0.3) $);
\path[draw, <-, bend left] ($ (B) + (20:0.3) $) to node[midway, above] {} ($ (C) + (160:0.3) $);
\path[draw, ->, bend right=120, looseness=50] ($ (B) + (90:0.3) $) to node[pos=0.4, right] {} ($ (B) + (110:0.3) $);
\path[draw, ->, bend right=120, looseness=50] ($ (B) + (150:0.3) $) to node[pos=0.4, left] {} ($ (B) + (170:0.3) $);
\end{tikzpicture}
\caption{This figure depicts the double quiver $ \qQ $ which we study as an example.}
\label{fig:intro-example-quiver}
\end{figure}

\paragraph*{Step 1}
For the first step, we explain how to find the symplectic codimension-2 leaves in $ \M(Q, α) $. To accomplish this, we first calculate the Cartan matrix of $ \qQ $, which reads
\begin{align*}
(-, -) = \begin{pmatrix}
2 & -1 & 0 \\
-1 & 0 & -1 \\
0 & -1 & 2
\end{pmatrix}.
\end{align*}
According to Bellamy and Schedler, we find the codimension-2 leaves in a purely combinatorial fashion by searching for “isotropic decompositions” of $ α $. These are a certain class of decompositions $ α = n_1 β_1 + … + n_k β_k $ whose most important requirement is that their “local quiver” is a Kleinian quiver (extended Dynkin quiver).

We shall here provide the decomposition $ α = e_1 + e_2 + e_2 + e_3 $ and explain that it is isotropic. We start by regarding the local quiver of the decomposition. By definition, it has four vertices, corresponding to the four terms in the decomposition. The vertices are joined by arrows whose quantity is given by the negative of the Cartan pairing. We calculate $ (e_1, e_2) = -1 $ and $ (e_1, e_3) = 0 $ and $ (e_2, e_3) = -1 $. The local quiver is depicted in \autoref{fig:intro-example-local} and we recognize it as the Kleinian $ A_3 $ quiver. Following Bellamy and Schedler, we conclude that $ α = e_1 + e_2 + e_2 + e_3 $ is indeed an isotropic decomposition. One can show that this is the only isotropic decomposition of $ α $. In other words, the symplectic singularity $ \M(Q, α) $ has only one codimension-2 leaf, which we denote by $ L $.

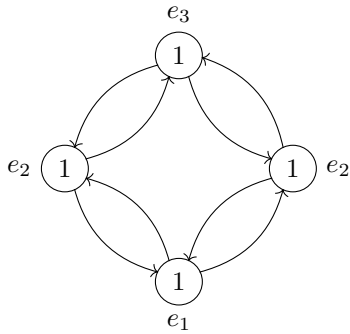
\begin{figure}
\centering
\begin{tikzpicture}
\path (0, 0) node[circle, draw] (A) {1} (1.5, 1.5) node[circle, draw] (B) {1} (0, 3) node[circle, draw] (C) {1} (-1.5, 1.5) node[circle, draw] (D) {1};
\path[draw, ->, bend right] ($ (A) + (25:0.3) $) to ($ (B) + (245:0.3) $);
\path[draw, <-, bend left] ($ (A) + (65:0.3) $) to ($ (B) + (205:0.3) $);
\path[draw, ->, bend right] ($ (B) + (115:0.3) $) to ($ (C) + (355:0.3) $);
\path[draw, <-, bend left] ($ (B) + (155:0.3) $) to ($ (C) + (295:0.3) $);
\path[draw, ->, bend right] ($ (C) + (205:0.3) $) to ($ (D) + (65:0.3) $);
\path[draw, <-, bend left] ($ (C) + (245:0.3) $) to ($ (D) + (25:0.3) $);
\path[draw, ->, bend right] ($ (D) + (295:0.3) $) to ($ (A) + (155:0.3) $);
\path[draw, <-, bend left] ($ (D) + (335:0.3) $) to ($ (A) + (115:0.3) $);
\path (A.south) node[below] {$ e_1 $} (B.east) node[right] {$ e_2 $} (C.north) node[above] {$ e_3 $} (D.west) node[left] {$ e_2 $};
\end{tikzpicture}
\caption{This figure depicts the local quiver associated to the decomposition $ α = e_1 + e_2 + e_2 + e_3 $. Evidently, the local quiver is isomorphic to the Kleinian $ A_3 $ quiver.}
\label{fig:intro-example-local}
\end{figure}

Let us describe the representations in the single codimension-2 leaf $ L $ explicitly. Following Bellamy and Schedler, these representations are precisely those which can be written in the form $ S_1 ⊕ S_2 ⊕ S_3 ⊕ S_4 $, where $ S_1 ∈ \M^{\simp} (Q, e_1) $, $ S_2 ∈ \M^{\simp} (Q, e_2) $, $ S_3 ∈ \M^{\simp} (Q, e_2) $ and $ S_4 ∈ \M^{\simp} (Q, e_3) $ with $ S_2 \not\cong S_3 $ are simple representations of dimensions $ e_1 $, $ e_2 $, $ e_2 $ and $ e_3 $, respectively. Following this characterization, the representations on the leaf $ L $ are depicted in \autoref{fig:intro-example-leaf}.

\begin{figure}
\centering
\begin{tikzpicture}
\path (0, 0) node[circle, draw] (A) {1} (0, 1.5) node[circle, draw] (B) {2} (1.5, 1.5) node[circle, draw] (C) {1};
\path[draw, ->, bend right] ($ (A) + (70:0.3) $) to node[midway, right] {$ 0 $} ($ (B) + (290:0.3) $);
\path[draw, <-, bend left] ($ (A) + (110:0.3) $) to node[midway, left] {$ 0 $}  ($ (B) + (250:0.3) $);
\path[draw, ->, bend right] ($ (B) + (340:0.3) $) to node[midway, below] {$ 0 $} ($ (C) + (200:0.3) $);
\path[draw, <-, bend left] ($ (B) + (20:0.3) $) to node[midway, above] {$ 0 $} ($ (C) + (160:0.3) $);
\path[draw, ->, bend right=120, looseness=50] ($ (B) + (90:0.3) $) to node[pos=0.4, right] {$ \begin{pmatrix} κ_1 & 0 \\ 0 & κ_2 \end{pmatrix} $} ($ (B) + (110:0.3) $);
\path[draw, ->, bend right=120, looseness=50] ($ (B) + (150:0.3) $) to node[pos=0.4, left] {$ \begin{pmatrix} κ_1^* & 0 \\ 0 & κ_2^* \end{pmatrix} $} ($ (B) + (170:0.3) $);
\end{tikzpicture}
\caption{This figure depicts the representations in the leaf $ L $. They are direct sums of representations of dimension vectors $ e_1 $, $ e_2 $, $ e_2 $ and $ e_3 $.}
\label{fig:intro-example-leaf}
\end{figure}
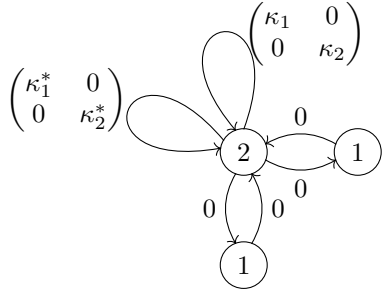

\paragraph*{Step 2}
For the second step, let us explain how to find a symplectic resolution for $ \M(Q, α) $. Geometric Invariant Theory tells us that we can find such a resolution by picking the quiver variety $ \M_θ (Q, α) $ where $ θ ∈ ℤ^{Q_0} $ is a generic stability parameter. The points in $ \M_θ (Q, α) $ are in one-to-one correspondence with $ θ $-polystable representations up to isomorphism. The resolution map $ π: \M_θ (Q, α) → \M(Q, α) $ is given by sending a $ θ $-polystable representation to its semisimplification.

We pick the stability parameter $ θ = (-1, -1, +3) $ and aim to identify representations in the fiber $ π^{-1} (L) $. By the abstract theory, we already know that for every $ x ∈ L $ the fiber $ π^{-1} (x) $ consists of three projective lines, which intersect according to the shape of the $ A_3 $ Dynkin diagram. By manual work, we find two of the three one-parameter families very quickly, which we have depicted in \autoref{fig:intro-example-fiber}. Inspection shows that these two families correspond to the two opposite exterior nodes of the $ A_3 $ Dynkin diagram.

\begin{figure}
\begin{tikzpicture}
\path[fill] (0, 0) coordinate (N1) circle[radius=0.1] (4, 0) coordinate (N2) circle[radius=0.1] (8, 0) coordinate (N3) circle[radius=0.1];
\path[draw] ($ (N1) + (right:0.3) $) -- ($ (N2) + (left:0.3) $);
\path[draw] ($ (N2) + (right:0.3) $) -- ($ (N3) + (left:0.3) $);
\begin{scope}[shift={(0, -5.5)}]
\path (0, 0) node[circle, draw] (A) {1} (0, 2.5) node[circle, draw] (B) {2} (2.5, 2.5) node[circle, draw] (C) {1};
\path[draw, ->, bend right] ($ (A) + (70:0.3) $) to node[pos=0.4, right] {\small $ \begin{pmatrix}1 \\ 0 \end{pmatrix} $} ($ (B) + (290:0.3) $);
\path[draw, <-, bend left] ($ (A) + (110:0.3) $) to node[midway, left] {\small $ \begin{pmatrix} 0 & κ_1^* - κ_2^* \end{pmatrix} $}  ($ (B) + (250:0.3) $);
\path[draw, ->, bend right] ($ (B) + (340:0.3) $) to node[midway, below] {$ \begin{pmatrix} * & 1 \end{pmatrix} $} ($ (C) + (200:0.3) $);
\path[draw, <-, bend left] ($ (B) + (20:0.3) $) to node[midway, above] {$ 0 $} ($ (C) + (160:0.3) $);
\path[draw, ->, bend right=120, looseness=50] ($ (B) + (90:0.3) $) to node[pos=0.4, right] {$ \begin{pmatrix} κ_1 & 1 \\ 0 & κ_2 \end{pmatrix} $} ($ (B) + (110:0.3) $);
\path[draw, ->, bend right=120, looseness=50] ($ (B) + (150:0.3) $) to node[pos=0.4, left] {$ \begin{pmatrix} κ_1^* & 0 \\ 0 & κ_2^* \end{pmatrix} $} ($ (B) + (170:0.3) $);
\end{scope}
\begin{scope}[shift={(8, -5.5)}]
\path (0, 0) node[circle, draw] (A) {1} (0, 2.5) node[circle, draw] (B) {2} (2.5, 2.5) node[circle, draw] (C) {1};
\path[draw, ->, bend right] ($ (A) + (70:0.3) $) to node[pos=0.4, right] {\small $ \begin{pmatrix} 0 \\ 1 \end{pmatrix} $} ($ (B) + (290:0.3) $);
\path[draw, <-, bend left] ($ (A) + (110:0.3) $) to node[midway, left] {\small $ \begin{pmatrix} 0 & κ_2^* - κ_1^* \end{pmatrix} $}  ($ (B) + (250:0.3) $);
\path[draw, ->, bend right] ($ (B) + (340:0.3) $) to node[midway, below] {$ \begin{pmatrix} * & 1 \end{pmatrix} $} ($ (C) + (200:0.3) $);
\path[draw, <-, bend left] ($ (B) + (20:0.3) $) to node[midway, above] {$ 0 $} ($ (C) + (160:0.3) $);
\path[draw, ->, bend right=120, looseness=50] ($ (B) + (90:0.3) $) to node[pos=0.4, right] {$ \begin{pmatrix} κ_1 & 0 \\ 1 & κ_2 \end{pmatrix} $} ($ (B) + (110:0.3) $);
\path[draw, ->, bend right=120, looseness=50] ($ (B) + (150:0.3) $) to node[pos=0.4, left] {$ \begin{pmatrix} κ_1^* & 0 \\ 0 & κ_2^* \end{pmatrix} $} ($ (B) + (170:0.3) $);
\end{scope}
\end{tikzpicture}
\caption{This figure depicts two of the three one-parameter families of $ θ $-stable representations which lie above the codimension-2 leaf $ L $. It is easily verified that the depicted representations are indeed all non-isomorphic and semisimplify to the representation given by $ S_2 = (κ_1, κ_1^*) $ and $ S_3 = (κ_2, κ_2^*) $. The thick lines visualize the $ A_3 $ Dynkin diagram. The two one-parameter families precisely correspond to the two outer vertices of the Dynkin diagram.}
\label{fig:intro-example-fiber}
\end{figure}
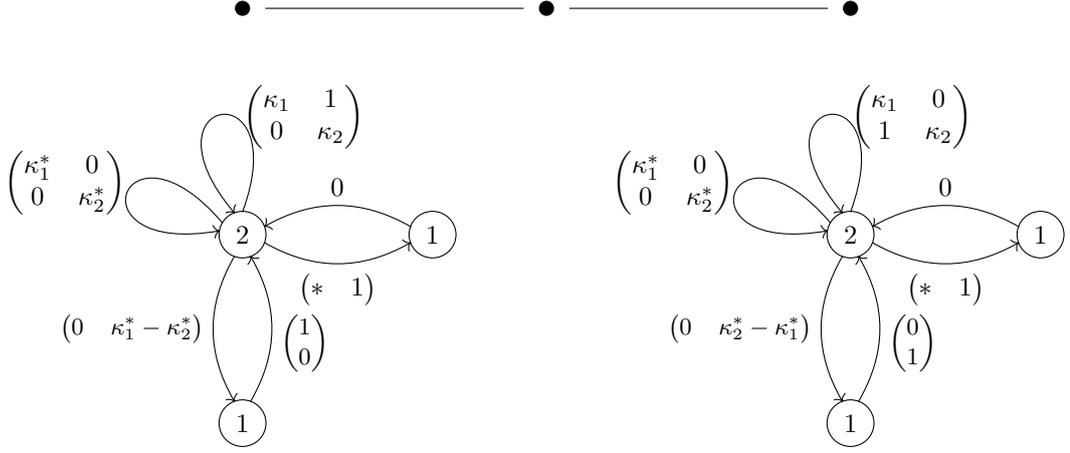

Let us give these representations a name. Let us write $ S_2 = (κ_1, κ_1^*) ∈ \M^{\simp} (Q, e_2) $ and $ S_3 = (κ_2, κ_2^*) ∈ \M^{\simp} (Q, e_2) $ with $ (κ_1, κ_1^*) ≠ (κ_2, κ_2^*) $. Then we denote the representation on the left by $ R_{κ_1, κ_1^*, κ_2, κ_2^*}^{(1)} $ and the right representation by $ R_{κ_1, κ_1^*, κ_2, κ_2^*}^{(3)} $. In this notation, we ignore the paramater denoted by $ * $. It is important to note that these representations indeed satisfy the preprojective condition, are $ θ $-stable and lie above the leaf. Moreover, representatives from the two families of representations are indeed non-isomorphic.

\paragraph*{Step 3}
For the third step, let us prove that we can pass from $ R_{κ_1, κ_1^*, κ_2, κ_2^*}^{(1)} $ to $ R_{κ_1, κ_1^*, κ_2, κ_2^*}^{(3)} $ by walking around in the leaf. Regard the space $ (ℂ^2 × ℂ^2) \setminus Δ $, where $ Δ ⊂ ℂ^2 × ℂ^2 $ denotes the diagonal. This space is not the same as the leaf, in fact the leaf is its quotient by $ C_2 $. In each case, the space $ (ℂ^2 × ℂ^2) \setminus Δ $ is connected and we pick a path $ γ: [0, 1] → (ℂ^2 × ℂ^2) \setminus Δ $ such that $ γ(0) = (κ_1, κ_1^*, κ_2, κ_2) $ and $ γ(1) = (κ_2, κ_2^*, κ_1, κ_1^*) $. Then the tracing the representation $ R_{κ_1, κ_1^*, κ_2, κ_2^*}^{(1)} $ along this path, we end up at $ R_{κ_2, κ_2^*, κ_1, κ_1^*}^{(1)} $. Note that we arrive at the same base point in the leaf again. Moreover, we easily convince ourselves that $ R_{κ_2, κ_2^*, κ_1, κ_1^*}^{(1)} $ is exactly the same as $ R_{κ_1, κ_1^*, κ_2, κ_2^*}^{(3)} $. This proves that the two outside vertices of the $ A_3 $ Dynkin diagram are indeed connected via the leaf.

\paragraph*{Step 4 and 5}
Let us draw the conclusions for the Namikawa-Weyl group of $ \M(Q, α) $. We have seen that for every point $ x ∈ L $, the two outer projective lines in $ π^{-1} (x) $ are connected to each other when we walk around in $ L $. This monodromy constitutes the nontrivial Dynkin automorphism $ τ $ of the $ A_3 $ Dynkin diagram. Recall that the classical Weyl group of $ A_3 $ type is $ W = S_6 $. Its invariant part $ W_1 = W^τ $ under the Dynkin automorphism is $ W_1 = S_2 \ltimes (ℤ/2ℤ)^2 $, also known as the Weyl group of $ C_2 $ type.

\paragraph*{Comparison with the main result}
We have just seen that the Namikawa-Weyl group of $ \M(Q, α) $ is the Weyl group of $ C_2 $ type. Let us compare this with our main result which determines the Namikawa-Weyl group explicitly. Indeed, our main result claims that the Dynkin graph automorphism associated with the single codimension-2 leaf is the nontrivial automorphism of the $ A_3 $-type Dynkin diagram, given that the labeled local quiver $ (Q', α', (β_1, …, β_4)) $ has an automorphism automorphism group of order two. This precisely agrees with the explicit calculation which we performed according to Namikawa's procedure.

%% file: app_functor/intro.tex
\section{Properties of the local-to-global functor}
\label{sec:functor}
In this section, we establish basic properties of the functor $ F $ which we constructed in \autoref{sec:main-functor}. In particular, we finish the proof of \autoref{th:result-functor} and \autoref{th:main-functor-stability}.

%% file: app_functor/honest.tex
\addtocontents{toc}{\SkipTocEntry}
\subsection{Honest modules}
\label{sec:functor-honest}
In this section, we recapitulate basic properties of honest modules.

\begin{lemma}
\label{th:functor-construction-honest}
Let $ Q $ be a quiver and let $ S_1, …, S_k ∈ \Mod_{Π_Q} $ be a sequence of modules. Then we have the following statements:
\begin{itemize}
\item Any twisted complex without shifts $ T = (M_1 ⊕ … ⊕ M_k, δ) ∈ \Tw\Mod_{Π_Q} $ is an honest module.
\item If $ S_1, …, S_k $ are simple, then any honest module in $ \Tw\{S_1, …, S_k\} $ is isomorphic to a shift-free twisted complex of $ S_1, …, S_k $.
\item If $ S_1, …, S_k $ are simple and $ X ∈ \Tw\{S_1, …, S_k\} $ is an honest module, and $ Y ⊂ X $ is a submodule, then $ Y $ can be realized as a subcomplex of $ X $.
\end{itemize}
\end{lemma}

\begin{proof}
Let us comment on all three statements after each other.

Regard the first statement. This is in fact a foundational property of $ A_∞ $-theory. In its simplest form, we have $ k = 2 $ and
\begin{equation*}
δ = \begin{pmatrix} 0 & e \\ 0 & 0 \end{pmatrix}.
\end{equation*}
Then $ T = (M_1 ⊕ M_2, δ) $ is simply quasi-isomorphic to the extension of $ M_2 $ by $ M_1 $ given by $ δ $. In particular, $ T $ is an honest module. The general case proceeds by induction.

Regard the second statement. Let $ X ∈ \Tw\{S_1, …, S_k\} $ be an honest module, quasi-isomorphic to a module $ M ∈ \Mod_{Π_Q} $. Let $ M_1, …, M_l $ be the simple composition factors of $ M $. Then $ M $ is quasi-isomorphic to a shift-free twisted complex $ Y = (M_1 ⊕ … ⊕ M_l, δ) $. In particular, $ X $ and $ Y $ are quasi-isomorphic and we have maps $ f: X → Y $ and $ g: Y → X $ of degree zero and maps $ h_X: X → X $ and $ h_Y: Y → Y $ of degree $ -1 $ such that $ μ^2 (f, g) = \id_Y + μ^1 (h_Y) $ and $ μ^2 (g, f) = \id_X + μ^1 (h_X) $. Since $ Y $ is shift-free, we necessarily have $ h_Y = 0 $. Thanks to the twisted complex formalism, the maps $ f $ and $ g $ are matrices of morphisms between the components of $ X $ and $ Y $, and $ μ^2 (f, g) $ is given by matrix multiplication. Let us split the twisted complex $ X $ into a shift-free part $ Y_1 $ and a shifted part $ Y_2 $:
\begin{equation*}
X = \big(X_1 ⊕ X_2, \begin{pmatrix} δ_{11} & δ_{12} \\ δ_{21} & δ_{22} \end{pmatrix} \big).
\end{equation*}
Let us also split $ f = (f_1 ~~ f_2) $ and $ g = (g_1, g_2) $ along this decomposition. Then we the composition $ μ^2 (f_2, g_2): Y → X_2 → Y $ vanishes, given that a morphism $ M_i → S_j [s] $ and $ S_j [s] → M_l $ can be nonzero but not both since the hom spaces of $ \Mod_{Π_Q} $ are concentrated in nonnegative degrees. We conclude that
\begin{equation*}
\id_Y = μ^2 (f, g) = μ^2 (f_1, g_1).
\end{equation*}
Since $ S_1, …, S_k $ and $ M_1, …, M_l $ are simples, the matrices $ f_1 $ and $ g_1 $ are scalar matrices, and we conclude they have equal size and are inverse to each other. We conclude that $ f_1: X_1 → Y $ and $ g_1: Y → X_1 $ provide an isomorphism between $ X_1 $ and $ Y $. Moreover, since there are no degree zero morphisms between non-isomorphic simples, we conclude that every summand $ M_i $ lies in $ \{S_1, …, S_k\} $ and hence $ X_1 $ is a shift-free twisted complex of $ S_1, …, S_k $. Finally, since $ X $ was assumed to be quasi-isomorphic to $ Y $, we conclude that $ X $ is quasi-isomorphic to $ X_1 $, a shift-free twisted complex of $ S_1, …, S_k $. This proves the second statement.

The proof of the third statement is similar. This finishes the proof.
\end{proof}

%% file: app_functor/construction.tex
\addtocontents{toc}{\SkipTocEntry}
\subsection{The functor construction}
\label{sec:functor-construction}
In this section, we provide the details of the functor construction and in particular provide the proof of \autoref{th:result-functor}.

\begin{proof}[Proof of \autoref{th:result-functor}]
Let us start by constructing the functor $ F $. Indeed, both categories $ \{\tilde S_1, …, \tilde S_k\} ⊂ \Mod_{Π_{Q'}} $ and $ \{S_1, …, S_k\} ⊂ \Mod_{Π_Q} $ share exactly the same $ A_∞ $-structure by \autoref{th:prelim-extdim} and \autoref{th:prelim-cyclicity}. The only difference lies in the endomorphism spaces, which are smaller in $ \Mod_{Π_{Q'}} $. Therefore on hom spaces, we simply send the basis vector $ v_{ij, l} $ in $ \Mod_{Π_{Q'}} $ to the corresponding basis vector $ v_{ij, l} $ in $ \Mod_{Π_Q} $, and likewise the identities and starred basic vectors. We immediately obtain that $ F $ is a strict unital faithful $ A_∞ $-embedding.

Let us now show that $ F $ sends honest modules with honest modules. We start with an honest module $ T ∈ \Tw\{\tilde S_1, …, \tilde S_k\} $. Then $ T $ is quasi-isomorphic to some module $ M $ within $ \Mod_{Π_Q} $. Take a Jordan-Hölder filtration of $ M $ with simple subquotients $ M_1, …, M_s $. This provides a quasi-isomorphism between $ M $ and a twisted complex $ M' = (M_1 ⊕ … ⊕ M_s, δ) $. Now that $ T $ and $ M' $ are quasi-isomorphic, we have maps $ f: T → M' $ and $ g: M' → T $ of degree zero and maps $ h_T: T → T $ and $ h_{M'}: M' → T $ of degree $ -1 $ such that $ μ^2 (f, g) = \id_{M'} + μ^1 (h_{M'}) $ and $ μ^2 (g, f) = \id_T + μ^1 (h_T) $. But since the negative hom spaces vanish, we have $ h_T = h_{M'} = 0 $ and we conclude that $ T $ and $ M' $ are in fact isomorphic. In particular, they share the same composition factors and we conclude $ M_i ∈ \{\tilde S_1, …, \tilde S_k\} $. Finally, we conclude that $ F(M') $ is a twisted complex of unshifted copies of the modules $ F(\tilde S_i) = S_i $ and is by \autoref{th:functor-construction-honest} an honest module. Since $ F(T) $ is quasi-isomorphic to $ F(M') $, we conclude that $ F(T) $ is an honest module as well. This shows that $ F $ sends honest modules to honest modules, and finishes the proof.
\end{proof}

%% file: app_functor/stability.tex
\addtocontents{toc}{\SkipTocEntry}
\subsection{Stability of the local Calabi-Yau-2 functor}
\label{sec:functor-stability}
In this section, we show that the functor $ F $ preserves (semi)stable objects and in particular we prove \autoref{th:main-functor-stability}. We start from a stability parameter $ θ $ on the global quiver $ (Q, α) $ and define a corresponding local stability parameter $ θ' $ on the local quiver $ (Q', α') $. By inspection into the twisted complexes, we easily prove that $ F $ reflects subobjects, and conclude that $ F $ sends $ θ' $-semistable, $ θ' $-polystable and $ θ' $-stable representations to $ θ $-semistable, $ θ $-polystable and $ θ' $-semistable representations.
\begin{center}
\begin{tikzpicture}
\path (0, 0) node[align=center] (A) {\textbf{Representations of $ Q' $} \\ $ θ' $-semistable \\ $ θ' $-polystable \\ $ θ' $-stable};
\path (8, 0) node[align=center] (B) {\textbf{Representations of $ Q $} \\ $ θ $-semistable \\ $ θ $-polystable \\ $ θ $-stable};
\path[draw, <->] ($ (A.east) + (right:0.5) $) to node[midway, above] {Functor $ F $} ($ (B.west) + (left:0.5) $);
\end{tikzpicture}
\end{center}

\begin{definition}
Let $ (Q, α) $ be a quiver setting with $ α ∈ Σ_{0, 0} $. Let $ L $ be a codimension-2 leaf with local quiver $ (Q', α') $ and roots $ β_1, …, β_k $. Let $ θ ∈ ℤ^{Q_0} $ be a stability parameter such that $ θ · α = 0 $. Then the \emph{local stability parameter} $ θ' ∈ ℤ^{Q'_0} $ at $ L $ is defined by
\begin{equation*}
θ' = (θ · β_1, …, θ · β_k).
\end{equation*}
\end{definition}

We are now ready to compare the stability properties of a representation $ X $ with its image $ F(X) $. By abuse of notation, we denote the functor $ \HTw\{\tilde S_1, …, \tilde S_k\} → \H\Tw\{S_1, …, S_k\} $ by $ F $ as well. Thanks to our cautious preparations, we shall also mix the language of twisted complexes and modules. For instance, when $ X $ is a twisted complex that is an honest module, then we also write $ X $ for that module and use module notation such as “$ ⊂ $”.

\begin{lemma}
\label{th:functor-stability-subrep}
Let $ (Q, α) $ be a quiver setting and $ S_1, …, S_k ∈ \Mod_{Π_Q} $ be simple representations. Let $ Q' $ be the local quiver and $ \tilde S_1, …, \tilde S_k $ its standard simples. Let $ θ ∈ ℤ^{Q_0} $ be a stability parameter and $ θ' $ its localized version. Let $ X, Y ∈ \Tw\{\tilde S_1, …, \tilde S_k\} $ be two honest modules. Then we have the following statements:
\begin{itemize}
\item We have $ \dim(X) · θ' = \dim F(X) · θ $.
\item $ Y $ is a subrepresentation of $ X $ if and only if $ F(Y) $ is a subrepresentation of $ F(X) $.
\item If $ Y' ⊂ F(X) $ is a subrepresentation, then $ Y' $ lies in the essential image of $ F $.
\item The functor $ F $ preserves direct sums.
\end{itemize}
\end{lemma}

\begin{proof}
We prove all four statements in sequence.

Regard the first statement. By definition of the twisted completion of $ F $, it suffices to prove the statement in case $ X = \tilde S_i $ is one of the standard simples. The statement then holds by the definition $ θ'_i = \dim(S_i) · θ $.

Regard the second statement. Let $ Y ⊂ X $ be a subrepresentation. By \autoref{th:functor-construction-honest}, we can realize $ Y $ as a subcomplex of $ X $. Since $ F $ is a strict functor, $ F(Y) $ is a subcomplex of $ F(X) $. Since both $ F(X) $ and $ F(Y) $ are honest modules, we conclude that $ F(Y) $ is a subrepresentation of $ F(X) $. This finishes the first direction. Let now $ F(Y) ⊂ F(X) $ be a subrepresentation. By \autoref{th:functor-construction-honest}, we can realize $ F(Y) $ as a subcomplex of $ F(X) $. Since $ F $ is a strict functor, $ Y $ is a subcomplex of $ X $. We conclude that $ Y $ is a subrepresentation of $ X $. This proves the second statement.

Regard the third statement. Let $ Y' ⊂ F(X) $ be a subrepresentation. By \autoref{th:functor-construction-honest}, we can realize $ Y' $ as subcomplex of $ F(X) $. In particular, as a subcomplex complex $ Y' $ contains no self-ext terms and we conclude that $ Y' $ lies in the image of $ F $.

The fourth statement holds by definition of $ F $ on twisted complexes. This finishes the proof.
\end{proof}

We are now ready to finish the proof of \autoref{th:main-functor-stability}.

\begin{proof}[Proof of \autoref{th:main-functor-stability}]
The proof is now an easy consequence of \autoref{th:functor-stability-subrep} by making use of the characterization of stability in terms of submodules. We shall spell out the details for the first statement only.

Assume that $ X $ is $ θ' $-semistable and let $ Y' ⊂ F(X) $ be a subrepresentation. By \autoref{th:functor-stability-subrep}, $ Y' $ lies in the essential image of $ F $. Let us write $ Y' = F(Y) $. Then $ F(Y) $ is a subrepresentation of $ F(X) $ and by \autoref{th:functor-stability-subrep} we conclude that $ Y $ is a subrepresentation of $ X $. By \autoref{th:functor-stability-subrep} and the $ θ' $-semistability of $ X $, we have $ 0 ≤ \dim(Y) · θ' = \dim(F(Y)) · θ = \dim(Y') · θ $. We conclude that $ F(X) $ is $ θ $-semistable.

Assume conversely that $ F(X) $ is $ θ $-semistable and let $ Y ⊂ X $ be a subrepresentation. By \autoref{th:functor-stability-subrep}, $ F(Y) $ is a subrepresentation of $ F(X) $. By $ θ $-semistability of $ F(X) $, we have $ 0 ≤ \dim(F(Y)) · θ = \dim(Y) · θ' $. Since $ Y $ was arbitrary, we conclude that $ X $ is $ θ $-semistable.

This proves the first statement. The second and third statements are analogous and we finish the proof.
\end{proof}

%% file: app_excfiber/intro.tex
\section{Analysis of the Kleinian exceptional fiber}
\label{sec:app-excfiber}
In this section, we substantiate the terse explanations in \autoref{sec:excfiber} and provide the full proof of \autoref{th:main-excfiber-action}. More specifically, we describe the exceptional fiber of partial and full resolutions of Kleinian singularities and describe the action of the quiver symmetries on the exceptional fiber. In specific cases, we make the exceptional fiber and the action fully explicit, for reference and to enhance intuition.

We do not make much reference in this section to the original quiver setting $ (Q, α) $ of which $ (Q', α') $ is the local quiver. However, we only need to treat those quiver automorphisms $ σ $ and stability parameters $ θ' $ which can actually appear as elements of the automorphism group of a local quiver and as localized stability parameter of an actual pseudo-generic stability parameter. For instance, the possible group symmetries $ σ $ are only those listed in \autoref{fig:symmetry-symmetry-possible}. Similarly, the possible stability parameters $ θ' $ are quite restricted, first and foremost by necessarily $ σ $-symmetric, but secondly by the fact that they come from localization of a pseudo-generic $ θ $. For instance, the combination of the symmetry $ σ $ and the stability parameter $ θ' $ depicted in \autoref{fig:app-excfiber-impossible} is impossible. This implies that many irregular combinations of $ σ $ and $ θ' $ need not be treated in order to prove \autoref{th:main-excfiber-action}.

\begin{figure}
\begin{tikzpicture}
\begin{scope}[shift={(4, 0)}]
\path (2, 0) node (A) {$ +1 $};
\path (0, 2) node (B) {$ 0 $};
\path (-2, 0) node (C) {$ 0 $};
\path (0, -2) node (D) {$ +1 $};
\path (0, 0) node (E) {$ -1 $};
\path[draw, ->] ($ (C) + (20:0.4) $) to node[midway, above] {$ $} ($ (E) + (160:0.4) $);
\path[draw, <-] ($ (C) + (340:0.4) $) to node[midway, below] {$ $} ($ (E) + (200:0.4) $);
\path[draw, ->] ($ (B) + (290:0.4) $) to node[midway, right] {$ $} ($ (E) + (70:0.4) $);
\path[draw, <-] ($ (B) + (250:0.4) $) to node[midway, left] {$ $} ($ (E) + (110:0.4) $);
\path[draw, ->] ($ (A) + (200:0.4) $) to node[midway, below] {$  $} ($ (E) + (340:0.4) $);
\path[draw, <-] ($ (A) + (160:0.4) $) to node[midway, above] {$  $} ($ (E) + (20:0.4) $);
\path[draw, ->] ($ (D) + (110:0.4) $) to node[midway, left] {$  $} ($ (E) + (250:0.4) $);
\path[draw, <-] ($ (D) + (70:0.4) $) to node[midway, right] {$ $} ($ (E) + (290:0.4) $);
\path[draw] \foreach \i in {A, B, C, D, E} {(\i) circle[radius=0.3]};
\path[draw, thick, double equal sign distance, <->] ($ (B.south west) + (left:0.2) $) to node[midway, left] {$ σ $} ($ (C.north east) + (right:0.2) $);
\end{scope}
\end{tikzpicture}
\caption{This combination of $ σ ∈ \Aut(Q', α') $ and $ θ' ∈ ℤ^{Q'_0} $ need not be checked for \autoref{th:main-excfiber-action}, because it can impossibly appear as element of the symmetry group of a local labeled quiver $ (Q', α', (β_1, …, β_k)) $ and localization of a pseudo-generic stability parameter, under the assumption that $ α ∈ Σ_{0, 0} $ is indivisible or $ (\gcd(α), p(\gcd(α)^{-1} α)) = (2, 2) $, and that $ θ $ is pseudo-generic. Indeed, having $ θ'_1 = θ'_2 = 0 $ as indicated in the picture would imply that $ θ · β_1 = θ · β_2 = 0 $. Since $ θ $ is pseudo-generic, this implies $ β_1 = β_1 = α/2 $, which is absurd since $ α $ is the sum of all five roots. This shows that the depicted combination of $ σ $ and $ θ $ is impossible.}
\label{fig:app-excfiber-impossible}
\end{figure}



We proceed by distinguishing two cases: In case $ θ' $ is generic, the argument is short and is treated in \autoref{sec:app-excfiber-cb} and \autoref{sec:app-excfiber-all}. In case $ θ' $ is not generic, the proof collapses into case distinctions, but we indicate how to proceed in \autoref{sec:app-excfiber-partial}. This finishes the proof of \autoref{th:main-excfiber-action}.

%% file: app_excfiber/generic_cb.tex
\addtocontents{toc}{\SkipTocEntry}
\subsection{The case of $ θ' = (-K, +1, …, +1) $}
\label{sec:app-excfiber-cb}
Let $ (Q', α') $ be a Kleinian quiver setting and $ θ' ∈ ℤ^{Q'_0} $ be the stability parameter $ θ' = (-K, +1, …, +1) $ chosen by Crawley-Boevey \cite{cb-kleinian}. It is our aim to describe the action of the potential symmetries of $ (Q', α') $ on the exceptional fiber $ E = π^{-1} (0) $ where $ π: \M_{θ'} (Q', α') → \M(Q', α') $.

Crawley-Boevey investigates the case of a Kleinian quiver $ (Q', α') $ with $ θ' = (-K, +1, …, +1) $ where the $ -K $ stands on the special vertex. Note that this is a generic stability parameter and $ \M_{θ'} (Q', α') $ is smooth. Crawley-Boevey shows that representations in the $ n $-many projective lines of the special fiber are distinguished by their \emph{socle}. The socle is by definition the sum of all simple subrepresentations:
\begin{equation*}
\soc M ≔ \sum_{\substack{N ⊂ M \\ \text{simple}}} N ~⊂M.
\end{equation*}
The special fiber consists of $ n $-many projective lines, and on the bulk of each projective line, the socle is one of the vertex simples. As we let an automorphism $ σ ∈ \Aut(Q', α') $ act on a representation, the socle permutes accordingly and we conclude that the action of $ σ $ on the Dynkin diagram is of the same degree as $ σ $ itself. This already proves the desired statement.


In the remainder of this section, we describe for reference and intuition the exceptional fiber together with the action of $ σ $ explicitly in the case of $ A_3 $ with $ θ' = (-3, +1, +1, 1) $, of $ D_4 $ with $ θ' = (-5, +1, +1, +1, +1) $ and $ D_{n≥5} $ with $ θ' = (-(2n-3), +1, …, +1) $.

\begin{figure}
\centering
\begin{tikzpicture}
\begin{scope}
\path (0, 1) node (A) {$ -3 $};
\path (1, 0) node (B) {$ +1 $};
\path (0, -1) node (C) {$ +1 $};
\path (-1, 0) node (D) {$ +1 $};
\path[draw] \foreach \i in {A, B, C, D} {(\i) circle[radius=0.3]};
\path[draw, ->, bend left] ($ (A) + (0:0.4) $) to ($ (B) + (90:0.4) $);
\path[draw, <-, bend right] ($ (A) + (315:0.4) $) to ($ (B) + (135:0.4) $);
\path[draw, ->, bend left] ($ (B) + (270:0.4) $) to ($ (C) + (0:0.4) $);
\path[draw, <-, bend right] ($ (B) + (225:0.4) $) to ($ (C) + (45:0.4) $);
\path[draw, ->, bend left] ($ (C) + (180:0.4) $) to ($ (D) + (270:0.4) $);
\path[draw, <-, bend right] ($ (C) + (135:0.4) $) to ($ (D) + (315:0.4) $);
\path[draw, ->, bend left] ($ (D) + (90:0.4) $) to ($ (A) + (180:0.4) $);
\path[draw, <-, bend right] ($ (D) + (45:0.4) $) to ($ (A) + (225:0.4) $);
\end{scope}
\begin{scope}[shift={(4, 0)}]
\path (2, 0) node (A) {$ +1 $};
\path (0, 2) node (B) {$ +1 $};
\path (-2, 0) node (C) {$ -5 $};
\path (0, -2) node (D) {$ +1 $};
\path (0, 0) node (E) {$ +1 $};
\path[draw, ->] ($ (C) + (20:0.4) $) to node[midway, above] {$ $} ($ (E) + (160:0.4) $);
\path[draw, <-] ($ (C) + (340:0.4) $) to node[midway, below] {$ $} ($ (E) + (200:0.4) $);
\path[draw, ->] ($ (B) + (290:0.4) $) to node[midway, right] {$ $} ($ (E) + (70:0.4) $);
\path[draw, <-] ($ (B) + (250:0.4) $) to node[midway, left] {$ $} ($ (E) + (110:0.4) $);
\path[draw, ->] ($ (A) + (200:0.4) $) to node[midway, below] {$  $} ($ (E) + (340:0.4) $);
\path[draw, <-] ($ (A) + (160:0.4) $) to node[midway, above] {$  $} ($ (E) + (20:0.4) $);
\path[draw, ->] ($ (D) + (110:0.4) $) to node[midway, left] {$  $} ($ (E) + (250:0.4) $);
\path[draw, <-] ($ (D) + (70:0.4) $) to node[midway, right] {$ $} ($ (E) + (290:0.4) $);
\path[draw] \foreach \i in {A, B, C, D, E} {(\i) circle[radius=0.3]};
\end{scope}
\begin{scope}[shift={(7, 0)}]
\path (0, 1) node (A) {$ +1 $};
\path (0, -1) node (B) {$ +1 $};
\path (1, 0) node (C) {$ +1 $};
\path (2, 0) node (D) {$ +1 $};
\path (3, 0) node (E) {$ … $};
\path (4, 0) node (F) {$ +1 $};
\path (5, 0) node (G) {$ +1 $};
\path (6, 1) node (H) {$ +1 $};
\path (6, -1) node (I) {$ -(2n-3) $};
\path[draw, <-] ($ (A) + (335:0.4) $) to node[pos=0.1, right] {$ $} ($ (C) + (115:0.4) $);
\path[draw, ->] ($ (A) + (295:0.4) $) to node[pos=0.6, left] {$  $} ($ (C) + (155:0.4) $);
\path[draw, ->] ($ (B) + (65:0.4) $) to node[pos=0.4, left] {$ $} ($ (C) + (205:0.4) $);
\path[draw, <-] ($ (B) + (25:0.4) $) to node[pos=0.4, right] {$ $} ($ (C) + (245:0.4) $);
\path[draw, ->] ($ (C) + (20:0.4) $) to node[midway, above] {$ $} ($ (D) + (160:0.4) $);
\path[draw, <-] ($ (C) + (340:0.4) $) to node[midway, below] {$ $} ($ (D) + (200:0.4) $);
\path[draw, ->] ($ (F) + (20:0.4) $) to node[midway, above] {$ $} ($ (G) + (160:0.4) $);
\path[draw, <-] ($ (F) + (340:0.4) $) to node[midway, below] {$ $} ($ (G) + (200:0.4) $);
\path[draw, ->] ($ (G) + (65:0.4) $) to node[pos=0.6, left] {$ $} ($ (H) + (205:0.4) $);
\path[draw, <-] ($ (G) + (25:0.4) $) to node[pos=0.6, right] {$ $} ($ (H) + (245:0.4) $);
\path[draw, <-] ($ (G) + (335:0.4) $) to node[pos=0.6, right] {$ $} ($ (I) + (115:0.4) $);
\path[draw, ->] ($ (G) + (295:0.4) $) to node[pos=0.6, left] {$  $} ($ (I) + (155:0.4) $);
\path[draw] \foreach \i in {A, B, C, D, F, G, H} {(\i) circle[radius=0.3]};
\begin{scope}[xscale=3] \path[draw] (I) circle[radius=0.3]; \end{scope}
\end{scope}
\end{tikzpicture}
\caption{This figure depicts three Kleinian quivers with choice of stability parameter. In this section we investigate their exceptional fibers.}
\label{fig:excfiber-smooth-quivers}
\end{figure}
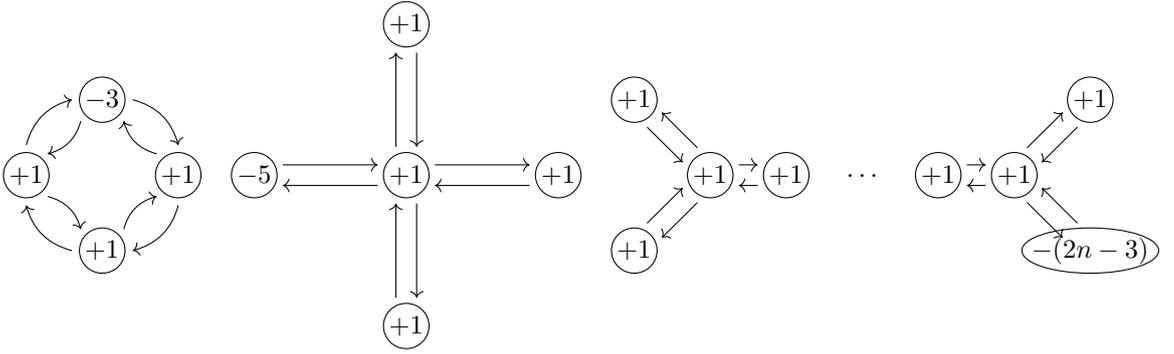

\addtocontents{toc}{\SkipTocEntry}
\subsubsection*{The case of $ A_3 $ with $ θ' = (-3, +1, +1, +1) $}
We shall explicitly describe the three 1-parameter families of $ ϑ' $-stable representations. A representation in the exceptional fiber $ π^{-1} (0) $ is then a representation which satisfies the preprojective condition, lies above zero, and is generated by the special vertex. The basic observation to find them in the $ A_3 $ case is that for a representation in the exceptional fiber, among every pair of arrows $ A_i $ and $ A_i^* $ the value of at least one vanishes. This makes it easy to guess and draw up the three 1-parameter families. We have depicted them in \autoref{fig:excfiber-smooth-A3}.

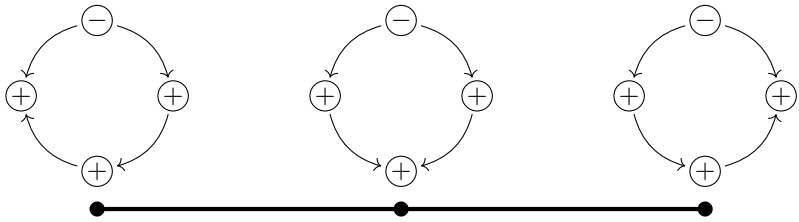
\begin{figure}
\centering
\begin{tikzpicture}
\begin{scope}
\path (0, 1) node (A) {$ - $};
\path (1, 0) node (B) {$ + $};
\path (0, -1) node (C) {$ + $};
\path (-1, 0) node (D) {$ + $};
\path[draw] \foreach \i in {A, B, C, D} {(\i) circle[radius=0.2]};
\path[draw, ->, bend left] (A) to (B);
\path[draw, ->, bend left] (B) to (C);
\path[draw, ->, bend left] (C) to (D);
\path[draw, <-, bend left] (D) to (A);
\end{scope}
\begin{scope}[shift={(right:4)}]
\path (0, 1) node (A) {$ - $};
\path (1, 0) node (B) {$ + $};
\path (0, -1) node (C) {$ + $};
\path (-1, 0) node (D) {$ + $};
\path[draw] \foreach \i in {A, B, C, D} {(\i) circle[radius=0.2]};
\path[draw, ->, bend left] (A) to (B);
\path[draw, ->, bend left] (B) to (C);
\path[draw, <-, bend left] (C) to (D);
\path[draw, <-, bend left] (D) to (A);
\end{scope}
\begin{scope}[shift={(right:8)}]
\path (0, 1) node (A) {$ - $};
\path (1, 0) node (B) {$ + $};
\path (0, -1) node (C) {$ + $};
\path (-1, 0) node (D) {$ + $};
\path[draw] \foreach \i in {A, B, C, D} {(\i) circle[radius=0.2]};
\path[draw, ->, bend left] (A) to (B);
\path[draw, <-, bend left] (B) to (C);
\path[draw, <-, bend left] (C) to (D);
\path[draw, <-, bend left] (D) to (A);
\end{scope}
\path[draw, ultra thick] (0, -1.5) -- (8, -1.5);
\path[fill] (0, -1.5) circle[radius=0.1];
\path[fill] (4, -1.5) circle[radius=0.1];
\path[fill] (8, -1.5) circle[radius=0.1];
\end{tikzpicture}
\caption{This figures depicts the three 1-parameter families of representations which make up the exceptional fiber of the $ A_3 $-quiver with the specific choice $ θ' = (-3, +1, +1, +1) $. The thick lines depict the $ A_3 $ Dynkin diagram, and each of its three vertices corresponds to one representation shape.}
\label{fig:excfiber-smooth-A3}
\end{figure}

\addtocontents{toc}{\SkipTocEntry}
\subsubsection*{The case of $ D_4 $ with $ θ' = (-1, -1, +5, -1, -1) $}
We shall explicitly describe the four 1-parameter families of $ ϑ' $-stable representations. The basic observation to find them is to use that all these representations have a filtration by the standard simples of $ Q' $. This implies that when replacing any nonzero numbers in the representation's matrices by arrows, the result is a non-circular diagram. In particular, no arrow leaves the vertex 3 of this diagram, and hence the value of the representation on the arrow from vertex 3 to 5 must be zero. These observations make it easy to draw up the four 1-parameter families, which we have depicted in \autoref{fig:excfiber-smooth-D4}.

\begin{figure}
\centering
\begin{tikzpicture}
\begin{scope}
\path (2, 0) node (A) {$ + $};
\path (0, 2) node (B) {$ + $};
\path (-2, 0) node (C) {$ - $};
\path (0, -2) node (D) {$ + $};
\path (0, 0) node (E) {$ + $};
\path[draw, ->] ($ (C) + (20:0.3) $) to node[midway, above] {$ e_1 $} ($ (E) + (160:0.3) $);
\path[draw, <-] ($ (C) + (340:0.3) $) to node[midway, below] {$ 0 $} ($ (E) + (200:0.3) $);
\path[draw, ->] ($ (B) + (290:0.3) $) to node[midway, right] {$ 0 $} ($ (E) + (70:0.3) $);
\path[draw, <-] ($ (B) + (250:0.3) $) to node[midway, left] {$ \small \begin{pmatrix} 1 & * \end{pmatrix} $} ($ (E) + (110:0.3) $);
\path[draw, ->] ($ (A) + (200:0.3) $) to node[midway, below] {$ e_2 $} ($ (E) + (340:0.3) $);
\path[draw, <-] ($ (A) + (160:0.3) $) to node[midway, above] {$ π_1 $} ($ (E) + (20:0.3) $);
\path[draw, ->] ($ (D) + (110:0.3) $) to node[midway, left] {$ -e_2 $} ($ (E) + (250:0.3) $);
\path[draw, <-] ($ (D) + (70:0.3) $) to node[midway, right] {$ π_1 $} ($ (E) + (290:0.3) $);
\path[draw] \foreach \i in {A, B, C, D, E} {(\i) circle[radius=0.2]};
\end{scope}
\begin{scope}[shift={(5, 7)}]
\path (2, 0) node (A) {$ + $};
\path (0, 2) node (B) {$ + $};
\path (-2, 0) node (C) {$ - $};
\path (0, -2) node (D) {$ + $};
\path (0, 0) node (E) {$ + $};
\path[draw, ->] ($ (C) + (20:0.3) $) to node[midway, above] {$ e_1 $} ($ (E) + (160:0.3) $);
\path[draw, <-] ($ (C) + (340:0.3) $) to node[midway, below] {$ 0 $} ($ (E) + (200:0.3) $);
\path[draw, ->] ($ (B) + (290:0.3) $) to node[midway, right] {$ e_2 $} ($ (E) + (70:0.3) $);
\path[draw, <-] ($ (B) + (250:0.3) $) to node[midway, left] {$ π_1 $} ($ (E) + (110:0.3) $);
\path[draw, ->] ($ (A) + (200:0.3) $) to node[midway, below] {$ 0 $} ($ (E) + (340:0.3) $);
\path[draw, <-] ($ (A) + (160:0.3) $) to node[midway, above] {$ \small \begin{pmatrix} 1 & * \end{pmatrix} $} ($ (E) + (20:0.3) $);
\path[draw, ->] ($ (D) + (110:0.3) $) to node[midway, left] {$ -e_2 $} ($ (E) + (250:0.3) $);
\path[draw, <-] ($ (D) + (70:0.3) $) to node[midway, right] {$ π_1 $} ($ (E) + (290:0.3) $);
\path[draw] \foreach \i in {A, B, C, D, E} {(\i) circle[radius=0.2]};
\end{scope}
\begin{scope}[shift={(10, 0)}]
\path (2, 0) node (A) {$ + $};
\path (0, 2) node (B) {$ + $};
\path (-2, 0) node (C) {$ - $};
\path (0, -2) node (D) {$ + $};
\path (0, 0) node (E) {$ + $};
\path[draw, ->] ($ (C) + (20:0.3) $) to node[midway, above] {$ e_1 $} ($ (E) + (160:0.3) $);
\path[draw, <-] ($ (C) + (340:0.3) $) to node[midway, below] {$ 0 $} ($ (E) + (200:0.3) $);
\path[draw, ->] ($ (B) + (290:0.3) $) to node[midway, right] {$ e_2 $} ($ (E) + (70:0.3) $);
\path[draw, <-] ($ (B) + (250:0.3) $) to node[midway, left] {$ π_1 $} ($ (E) + (110:0.3) $);
\path[draw, ->] ($ (A) + (200:0.3) $) to node[midway, below] {$ -e_2 $} ($ (E) + (340:0.3) $);
\path[draw, <-] ($ (A) + (160:0.3) $) to node[midway, above] {$ π_1 $} ($ (E) + (20:0.3) $);
\path[draw, ->] ($ (D) + (110:0.3) $) to node[midway, left] {$ 0 $} ($ (E) + (250:0.3) $);
\path[draw, <-] ($ (D) + (70:0.3) $) to node[midway, right] {$ \small \begin{pmatrix} 1 & * \end{pmatrix} $} ($ (E) + (290:0.3) $);
\path[draw] \foreach \i in {A, B, C, D, E} {(\i) circle[radius=0.2]};
\end{scope}
\begin{scope}[shift={(5, 0)}]
\path (2, 0) node (A) {$ + $};
\path (0, 2) node (B) {$ + $};
\path (-2, 0) node (C) {$ - $};
\path (0, -2) node (D) {$ + $};
\path (0, 0) node (E) {$ + $};
\path[draw, ->] ($ (C) + (20:0.3) $) to node[midway, above] {$ e_1 $} ($ (E) + (160:0.3) $);
\path[draw, <-] ($ (C) + (340:0.3) $) to node[midway, below] {$ 0 $} ($ (E) + (200:0.3) $);
\path[draw, ->] ($ (B) + (290:0.3) $) to node[midway, right] {$ αe_2 $} ($ (E) + (70:0.3) $);
\path[draw, <-] ($ (B) + (250:0.3) $) to node[midway, left] {$ π_1 $} ($ (E) + (110:0.3) $);
\path[draw, ->] ($ (A) + (200:0.3) $) to node[midway, below] {$ βe_2 $} ($ (E) + (340:0.3) $);
\path[draw, <-] ($ (A) + (160:0.3) $) to node[midway, above] {$ π_1 $} ($ (E) + (20:0.3) $);
\path[draw, ->] ($ (D) + (110:0.3) $) to node[midway, left] {$ γe_2 $} ($ (E) + (250:0.3) $);
\path[draw, <-] ($ (D) + (70:0.3) $) to node[midway, right] {$ π_1 $} ($ (E) + (290:0.3) $);
\path[draw] \foreach \i in {A, B, C, D, E} {(\i) circle[radius=0.2]};
\end{scope}
\path[draw, ultra thick] (0, 3) to (10, 3);
\path[draw, ultra thick] (5, 3) to (5, 4);
\path[fill] (0, 3) circle[radius=0.1];
\path[fill] (5, 3) circle[radius=0.1];
\path[fill] (10, 3) circle[radius=0.1];
\path[fill] (5, 4) circle[radius=0.1];
\end{tikzpicture}
\caption{This figures depicts the four 1-parameter families of representations which make up the exceptional fiber of the $ D_4 $-quiver with the specific choice $ θ' = (+1, +1, -5, +1, +1) $. In the representation in the middle, we suppose $ α + β + γ = 0 $, which gives a 1-parameter family after dividing out by isomorphisms. The thick lines depict the $ D_4 $ Dynkin diagram, and each of its four vertices corresponds to one representation shape.}
\label{fig:excfiber-smooth-D4}
\end{figure}
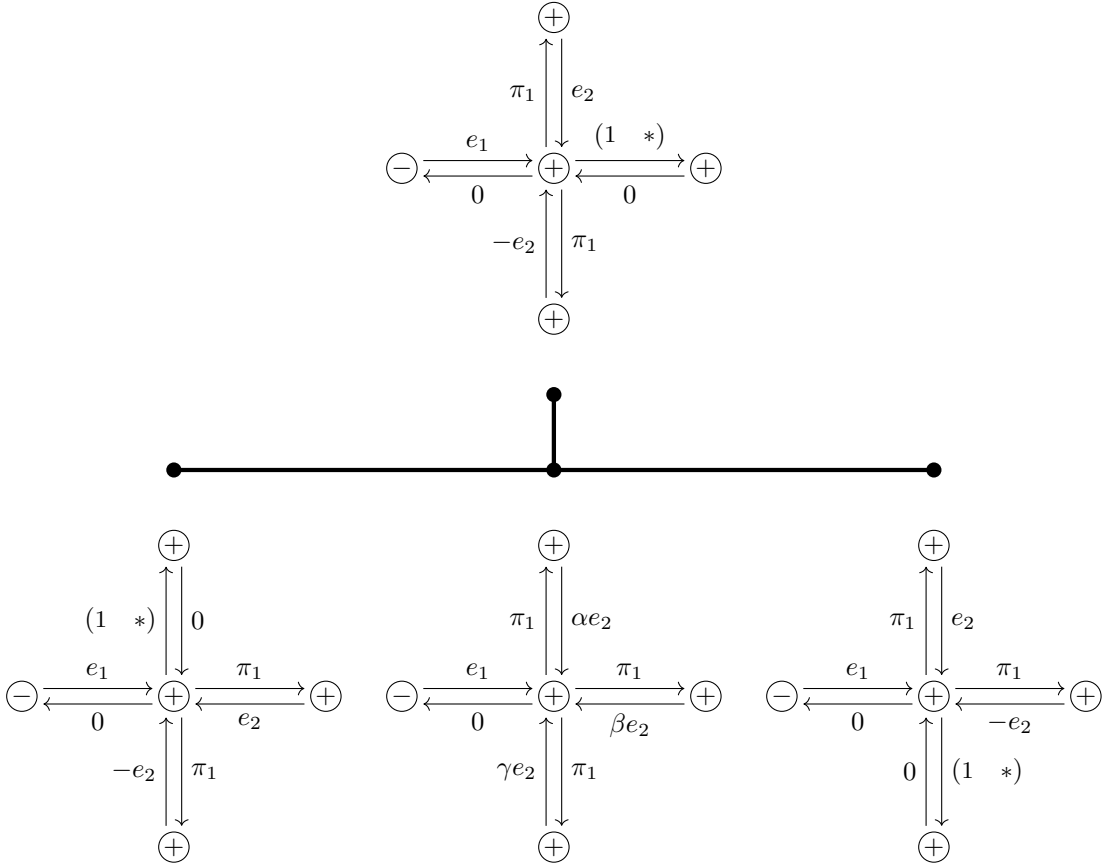

\addtocontents{toc}{\SkipTocEntry}
\subsubsection*{The case of $ D_{n≥5} $ with $ θ' = (+1, +1, +1, …, +1, +1, -(2n-3)) $}
We shall explicitly describe the $ n $-many 1-parameter families of $ ϑ' $-stable representations. Again, we use the basic observation that all these representations have a filtration by the standard simples of $ Q' $, and Crawley-Boevey's observation that all the vertices apart from vertex $ n $ appear as socles. We have depicted the $ n $-many families in \autoref{fig:excfiber-smooth-Dn}.

\begin{figure}
\centering
\begin{tikzpicture}
\begin{scope}[shift={(0, 5)}]
\path (0, 1) node (A) {$ + $};
\path (0, -1) node (B) {$ + $};
\path (1, 0) node (C) {$ + $};
\path (2, 0) node (D) {$ + $};
\path (3, 0) node (E) {$ … $};
\path (4, 0) node (F) {$ + $};
\path (5, 0) node (G) {$ + $};
\path (6, 1) node (H) {$ + $};
\path (6, -1) node (I) {$ - $};
\path[draw, <-] ($ (A) + (335:0.3) $) to node[pos=0.1, right] {$ (1 ~~ *) $} ($ (C) + (115:0.3) $);
\path[draw, ->] ($ (A) + (295:0.3) $) to node[pos=0.6, left] {$ 0 $} ($ (C) + (155:0.3) $);
\path[draw, ->] ($ (B) + (65:0.3) $) to node[pos=0.4, left] {$ e_2 $} ($ (C) + (205:0.3) $);
\path[draw, <-] ($ (B) + (25:0.3) $) to node[pos=0.3, right] {$ π_1 $} ($ (C) + (245:0.3) $);
\path[draw, ->] ($ (C) + (20:0.3) $) to node[midway, above] {$ \tiny \begin{pmatrix} 0 & 0 \\ 1 & 0 \end{pmatrix} $} ($ (D) + (160:0.3) $);
\path[draw, <-] ($ (C) + (340:0.3) $) to node[midway, below] {$ I_2 $} ($ (D) + (200:0.3) $);
\path[draw, ->] ($ (F) + (20:0.3) $) to node[midway, above] {$ \tiny \begin{pmatrix} 0 & 0 \\ 1 & 0 \end{pmatrix} $} ($ (G) + (160:0.3) $);
\path[draw, <-] ($ (F) + (340:0.3) $) to node[midway, below] {$ I_2 $} ($ (G) + (200:0.3) $);
\path[draw, ->] ($ (G) + (65:0.3) $) to node[pos=0.6, left] {$ π_1 $} ($ (H) + (205:0.3) $);
\path[draw, <-] ($ (G) + (25:0.3) $) to node[pos=0.6, right] {$ e_2 $} ($ (H) + (245:0.3) $);
\path[draw, <-] ($ (G) + (335:0.3) $) to node[pos=0.6, right] {$ e_1 $} ($ (I) + (115:0.3) $);
\path[draw, ->] ($ (G) + (295:0.3) $) to node[pos=0.6, left] {$ 0 $} ($ (I) + (155:0.3) $);
\path[draw] \foreach \i in {A, B, C, D, F, G, H, I} {(\i) circle[radius=0.2]};
\end{scope}
\begin{scope}[shift={(0, -5)}]
\path (0, 1) node (A) {$ + $};
\path (0, -1) node (B) {$ + $};
\path (1, 0) node (C) {$ + $};
\path (2, 0) node (D) {$ + $};
\path (3, 0) node (E) {$ … $};
\path (4, 0) node (F) {$ + $};
\path (5, 0) node (G) {$ + $};
\path (6, 1) node (H) {$ + $};
\path (6, -1) node (I) {$ - $};
\path[draw, <-] ($ (A) + (335:0.3) $) to node[pos=0.1, right] {$ π_1 $} ($ (C) + (115:0.3) $);
\path[draw, ->] ($ (A) + (295:0.3) $) to node[pos=0.6, left] {$ e_2 $} ($ (C) + (155:0.3) $);
\path[draw, ->] ($ (B) + (65:0.3) $) to node[pos=0.4, left] {$ 0 $} ($ (C) + (205:0.3) $);
\path[draw, <-] ($ (B) + (25:0.3) $) to node[pos=0.3, right] {$ (1 ~~ *) $} ($ (C) + (245:0.3) $);
\path[draw, ->] ($ (C) + (20:0.3) $) to node[midway, above] {$ \tiny \begin{pmatrix} 0 & 0 \\ 1 & 0 \end{pmatrix} $} ($ (D) + (160:0.3) $);
\path[draw, <-] ($ (C) + (340:0.3) $) to node[midway, below] {$ I_2 $} ($ (D) + (200:0.3) $);
\path[draw, ->] ($ (F) + (20:0.3) $) to node[midway, above] {$ \tiny \begin{pmatrix} 0 & 0 \\ 1 & 0 \end{pmatrix} $} ($ (G) + (160:0.3) $);
\path[draw, <-] ($ (F) + (340:0.3) $) to node[midway, below] {$ I_2 $} ($ (G) + (200:0.3) $);
\path[draw, ->] ($ (G) + (65:0.3) $) to node[pos=0.6, left] {$ π_1 $} ($ (H) + (205:0.3) $);
\path[draw, <-] ($ (G) + (25:0.3) $) to node[pos=0.6, right] {$ e_2 $} ($ (H) + (245:0.3) $);
\path[draw, <-] ($ (G) + (335:0.3) $) to node[pos=0.6, right] {$ e_1 $} ($ (I) + (115:0.3) $);
\path[draw, ->] ($ (G) + (295:0.3) $) to node[pos=0.6, left] {$ 0 $} ($ (I) + (155:0.3) $);
\path[draw] \foreach \i in {A, B, C, D, F, G, H, I} {(\i) circle[radius=0.2]};
\end{scope}
\begin{scope}[shift={(10, -2)}, transform canvas={scale=0.7}]
\path (0, 1) node (A) {$ + $};
\path (0, -1) node (B) {$ + $};
\path (1, 0) node (C) {$ + $};
\path (2, 0) node (D) {$ + $};
\path (3, 0) node (E) {$ … $};
\path (4, 0) node (F) {$ + $};
\path (5, 0) node (G) {$ + $};
\path (6, 1) node (H) {$ + $};
\path (6, -1) node (I) {$ - $};
\path[draw, <-] ($ (A) + (335:0.3) $) to node[pos=0.1, right] {$ π_1 $} ($ (C) + (115:0.3) $);
\path[draw, ->] ($ (A) + (295:0.3) $) to node[pos=0.6, left] {$ α e_2 $} ($ (C) + (155:0.3) $);
\path[draw, ->] ($ (B) + (65:0.3) $) to node[pos=0.4, left] {$ (-α-1) e_2 $} ($ (C) + (205:0.3) $);
\path[draw, <-] ($ (B) + (25:0.3) $) to node[pos=0.3, right] {$ π_1 $} ($ (C) + (245:0.3) $);
\path[draw, ->] ($ (C) + (20:0.3) $) to node[midway, above] {$ I_2 $} ($ (D) + (160:0.3) $);
\path[draw, <-] ($ (C) + (340:0.3) $) to node[midway, below] {$ \tiny \begin{pmatrix} 0 & 0 \\ 1 & 0 \end{pmatrix} $} ($ (D) + (200:0.3) $);
\path[draw, ->] ($ (F) + (20:0.3) $) to node[midway, above] {$ I_2 $} ($ (G) + (160:0.3) $);
\path[draw, <-] ($ (F) + (340:0.3) $) to node[midway, below] {$ \tiny \begin{pmatrix} 0 & 0 \\ 1 & 0 \end{pmatrix} $} ($ (G) + (200:0.3) $);
\path[draw, ->] ($ (G) + (65:0.3) $) to node[pos=0.6, left] {$ π_1 $} ($ (H) + (205:0.3) $);
\path[draw, <-] ($ (G) + (25:0.3) $) to node[pos=0.6, right] {$ e_2 $} ($ (H) + (245:0.3) $);
\path[draw, <-] ($ (G) + (335:0.3) $) to node[pos=0.6, right] {$ e_1 $} ($ (I) + (115:0.3) $);
\path[draw, ->] ($ (G) + (295:0.3) $) to node[pos=0.6, left] {$ 0 $} ($ (I) + (155:0.3) $);
\path[draw] \foreach \i in {A, B, C, D, F, G, H, I} {(\i) circle[radius=0.2]};
\end{scope}
\begin{scope}[shift={(8, 5)}]
\path (0, 1) node (A) {$ + $};
\path (0, -1) node (B) {$ + $};
\path (1, 0) node (C) {$ + $};
\path (2, 0) node (D) {$ + $};
\path (3, 0) node (E) {$ … $};
\path (4, 0) node (F) {$ + $};
\path (5, 0) node (G) {$ + $};
\path (6, 1) node (H) {$ + $};
\path (6, -1) node (I) {$ - $};
\path[draw, <-] ($ (A) + (335:0.3) $) to node[pos=0.1, right] {$ π_1 $} ($ (C) + (115:0.3) $);
\path[draw, ->] ($ (A) + (295:0.3) $) to node[pos=0.6, left] {$ e_2 $} ($ (C) + (155:0.3) $);
\path[draw, ->] ($ (B) + (65:0.3) $) to node[pos=0.4, left] {$ e_2 $} ($ (C) + (205:0.3) $);
\path[draw, <-] ($ (B) + (25:0.3) $) to node[pos=0.3, right] {$ π_1 $} ($ (C) + (245:0.3) $);
\path[draw, ->] ($ (C) + (20:0.3) $) to node[midway, above] {$ \tiny \begin{pmatrix} 0 & 0 \\ 0 & 1 \end{pmatrix} $} ($ (D) + (160:0.3) $);
\path[draw, <-] ($ (C) + (340:0.3) $) to node[midway, below] {$ \tiny \begin{pmatrix} 1 & 0 \\ 0 & 0 \end{pmatrix} $} ($ (D) + (200:0.3) $);
\path[draw, ->] ($ (F) + (20:0.3) $) to node[midway, above] {$ \tiny \begin{pmatrix} 0 & 0 \\ 0 & 1 \end{pmatrix} $} ($ (G) + (160:0.3) $);
\path[draw, <-] ($ (F) + (340:0.3) $) to node[midway, below] {$ \tiny \begin{pmatrix} 1 & 0 \\ 0 & 0 \end{pmatrix} $} ($ (G) + (200:0.3) $);
\path[draw, ->] ($ (G) + (65:0.3) $) to node[pos=0.8, left] {$ (1 ~~ *) $} ($ (H) + (205:0.3) $);
\path[draw, <-] ($ (G) + (25:0.3) $) to node[pos=0.6, right] {$ 0 $} ($ (H) + (245:0.3) $);
\path[draw, <-] ($ (G) + (335:0.3) $) to node[pos=0.6, right] {$ e_1 $} ($ (I) + (115:0.3) $);
\path[draw, ->] ($ (G) + (295:0.3) $) to node[pos=0.6, left] {$ 0 $} ($ (I) + (155:0.3) $);
\path[draw] \foreach \i in {A, B, C, D, F, G, H, I} {(\i) circle[radius=0.2]};
\end{scope}
\path[draw, ultra thick] (3, 0) -- (5, 0);
\path[draw] (5.5, 0) node {\Large $ \cdots $};
\path[draw, ultra thick] (6, 0) -- (8, 0);
\path[draw, ultra thick] (8, 0) -- (10, 3);
\path[draw, ultra thick] (1, 3) -- (3, 0);
\path[draw, ultra thick] (1, -3) -- (3, 0);
\path[fill] (1, 3) circle[radius=0.1];
\path[fill] (1, -3) circle[radius=0.1];
\path[fill] (3, 0) circle[radius=0.1];
\path[fill] (8, 0) circle[radius=0.1];
\path[fill] (10, 3) circle[radius=0.1];
\end{tikzpicture}
\caption{This figures depicts the $ n $-many 1-parameter families of representations which make up the exceptional fiber of the $ D_{n≥5} $-quiver with the specific choice $ θ' = (+1, +1, +1, …, +1, -(2n-3)) $.The thick lines depict the $ D_n $ Dynkin diagram, and each of its $ n $ vertices corresponds to one representation shape.}
\label{fig:excfiber-smooth-Dn}
\end{figure}

%% file: app_excfiber/generic_all.tex
\addtocontents{toc}{\SkipTocEntry}
\subsection{The case of arbitrary generic $ θ $}
\label{sec:app-excfiber-all}
Let $ (Q', α') $ be a Kleinian quiver setting, let $ σ ∈ \Aut(Q', α') $ be a primitive group element from the list of \autoref{fig:symmetry-symmetry-possible} and let $ θ' ∈ ℤ^{Q'_0} $ be any stability parameter which satisfies $ θ' · α' = 0 $ and is symmetric with respect to $ σ $. Assume that $ θ' $ is generic. It is our task to identify the action of $ σ $ on the exceptional fiber $ E ≔ π_{θ'}^{-1} (0) $ with $ π_{θ'}: \M_{θ'} (Q', α') → \M(Q', α') $.

\begin{remark}
\label{rem:app-exfiber-all-remark}
The specific assumption that the negative entry of $ θ_0' $ be on the special vertex is unnecessary, and can be weakened to requiring only that the negative entry of $ θ_0' $ be one of the “outer vertices”, by which we mean any of the vertices of $ Q' $ obtained from the special vertex by means of a quiver automorphism. For instance, for the Kleinian $ A_n $ quiver every single vertex is an “outer vertex” and for the $ D_n $ quiver all the four vertices marked with $ α'_i = 1 $ are “outer vertices”.
\end{remark}

Thanks to \autoref{sec:app-excfiber-cb}, we can suppose we have checked the claim already for the specific symmetric $ θ_0' = (-K, +1, …, +1) $, and thanks to \autoref{rem:app-exfiber-all-remark} we can assume that we have also checked the claim already for those versions of $ θ_0' $ where the negative entry stands on any of the “outer vertices”. Therefore we can assume that $ θ_0' $ is also symmetric with respect to $ σ $. Indeed, in the $ A_3 $ case, by the primitivity assumption on $ σ $ we can put the $ -3 $ on one of the vertices not affected by $ σ $ and put $ +1 $ on the other three vertices, which renders $ θ_0' $ indeed $ σ $-symmetric. In the $ D_4 $ case, since $ θ' $ is assumed to be generic, not all four outer vertices can be involved in $ σ $, hence we can put $ -5 $ on one of the fixed outer vertices and $ +1 $ on all other vertices, which renders $ θ_0' $ indeed $ σ $-symmetric. In the $ D_{n≥5} $ case, by the primitivity assumption on $ σ $, at most one pair of outer vertices is involved in $ σ $, and we can put $ -(2n-3) $ on one of the outer fixed vertices and $ +1 $ on all other vertices, rendering $ θ_0' $ indeed $ σ $-symmetric. Finally, in every case we have achieved that $ θ_0' $ is $ σ $-symmetric.

Let us now explain why the action of $ σ $ on the special fiber of $ \M_{θ'_0} (Q', α') $ agrees with the action on the special fiber of $ \M_{θ'} (Q', α') $, up to conjugation. We start by regarding the map $ φ: \M(Q', α')  → \M(Q', α') $ which flips any representation by $ σ $. Since $ θ' $ is assumed generic, both $ \M_{θ'} (Q', α') $ and $ \M_{θ'_0} (Q', α') $ are minimal resolutions of $ \M(Q', α') $ and we obtain two lifts $ \tilde{φ}: \M_{θ'} (Q', α') → \M_{θ'} (Q', α') $ and $ \tilde{φ_0}: \M_{θ'_0} (Q', α') → \M_{θ'_0} (Q', α') $ of $ φ $, as well as an isomorphism $ ψ: \M_{θ'} (Q', α') \isoto \M_{θ'_0} (Q', α') $ over $ \M(Q', α') $. All of these maps fit into the following diagram:

\begin{center}
\begin{tikzcd}
\M_{θ'} \arrow{dd} \arrow{rrr}{\sim}[swap]{\tilde{φ}} & & & \M_{θ'} \arrow{dd} \\
& \M_{θ'_0} \arrow[ul, "ψ^{-1}"', "\sim"] \arrow[dl] \arrow[r, "\tilde{φ_0}"] & \M_{θ'_0} \arrow[dr] \arrow[ru, "ψ", "\sim"'] & \\
\M \arrow[rrr, "φ"] &&& \M
\end{tikzcd}
\end{center}
The left triangle and the right triangle commute, because of the chosen isomorphism between the two resolutions. The lower rectangle commutes because we have chosen $ \tilde{φ_0} $ as a lift. In particular, the composition map $ ψ ∘ \tilde{φ_0} ∘ ψ^{-1} $ is a lift of $ φ $. Since also $ \tilde{φ} $ is a lift of $ φ $, we have that $ ψ ∘ \tilde{φ_0} ∘ ψ^{-1} = \tilde{φ} $. This proves the claim.

%% file: app_excfiber/partial_case.tex
\addtocontents{toc}{\SkipTocEntry}
\subsection{The case of non-generic $ θ' $}
\label{sec:app-excfiber-partial}
In this section, we investigate the exceptional fiber of the remaining cases of partial resolutions of Kleinian singularities. We start from a Kleinian quiver $ (Q', α') $ of type $ A_3 $, $ D_4 $ or $ D_{n≥5} $, one of the primitive group elements $ σ ∈ \Aut(Q', α') $ listed in \autoref{fig:symmetry-symmetry-possible} and a stability parameter $ θ' ∈ ℤ^{Q'_0} $ with $ θ' · α' = 0 $. We assume that $ θ' $ is not generic, but that $ θ' ≠ 0 $. It is our task to analyze the special fiber $ π_{θ'}^{-1} (0) $ of $ π_{θ'}: \M_{θ'} (Q', α') → \M(Q', α') $ and identify the action of $ σ $ on the special fiber.

As a first step, we decide to treat only two specific cases. These two specific cases are the $ A_3 $ quiver with stability parameter $ θ' = (+1, -1, +1, -1) $ and the $ D_4 $ quiver with stability parameter $ θ' = (-1, -1, -1, -1, +2) $. In both cases $ \M_{θ'} (Q', α') $ is not smooth. Instead, as we shall see, the exceptional fiber contains a single projective line. The singular locus consists of two or three points on this projective line, respectively. In this sense, these two partial resolutions are as singular as possible. We therefore skip all the other cases and focus on analyzing the special fiber and the action of $ σ $ for these two specific cases.

The two quivers with stability parameter are depicted in \autoref{fig:excfiber-partial-A3} and \ref{fig:excfiber-partial-D4}. We start with the observation that in both the cases $ \M_{θ'} (Q', α') $ is not smooth. Indeed, we have $ α' ∈ Σ_{0, 0} $ but the stability parameter $ θ' $ is not generic in the sense that $ α' $ is divisible by two and we have $ θ' · (α'/2) = 0 $. This means that there are singular points in $ \M_{θ'} (Q', α') $. By the criterion of Bellamy and Schedler, the singular points are those representations which are $ θ' $-polystable but not $ θ' $-stable. This provides us with a cue to find the representations explicitly.

By the theory of Kleinian singularities we expect that a partial resolution of a Kleinian singularity has several singular points, connected by projective lines in between. As the singular points are blown up, new projective lines and potentially new singular points appear, until there are $ n $-many projective lines, intersecting in the shape of the Dynkin diagram. Looking backwards, the exceptional fiber of a partial resolution can be described as $ n $-many projective lines in the shape of the Dynkin diagram, with some of them contracted to points. This is precisely the pattern that we shall identify in the two specific $ A_3 $ and $ D_4 $ quivers.

Quick inspection provides us with the desired results:

\subsubsection*{The case of $ A_3 $ with $ θ' = (+1, -1, +1, -1) $}
We find one 1-paramater family of representations and two singular points, depicted in \autoref{fig:excfiber-partial-A3}. It is obvious that the two singular points are indeed distinct. We conclude that the single projective line in the exceptional fiber of $ \M_{θ'} (Q', α') $ corresponds to the middle node of the Dynkin diagram, while the two singular points correspond to the outer two nodes of the Dynkin diagram.

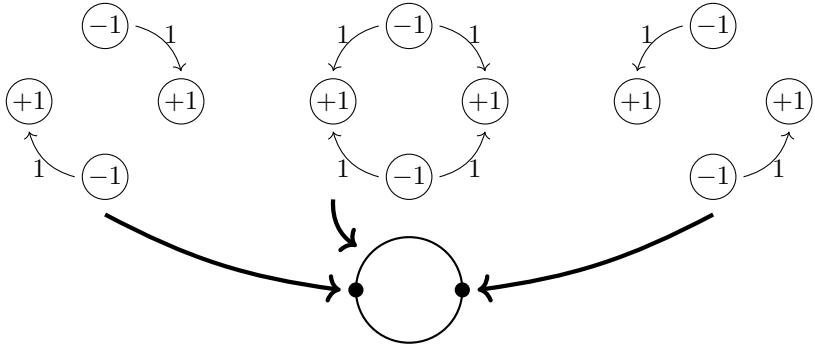
\begin{figure}
\centering
\begin{tikzpicture}
\begin{scope}
\path (0, 1) node (A) {$ -1 $};
\path (1, 0) node (B) {$ +1 $};
\path (0, -1) node (C) {$ -1 $};
\path (-1, 0) node (D) {$ +1 $};
\path[draw] \foreach \i in {A, B, C, D} {(\i) circle[radius=0.3]};
\path[draw, ->, bend left] ($ (A) + (0:0.4) $) to node[pos=0.3, right] {1} ($ (B) + (90:0.4) $);
\path[draw, ->, bend left] ($ (C) + (180:0.4) $) to node[pos=0.3, left] {1} ($ (D) + (270:0.4) $);
\end{scope}
\begin{scope}[shift={(4, 0)}]
\path (0, 1) node (A) {$ -1 $};
\path (1, 0) node (B) {$ +1 $};
\path (0, -1) node (C) {$ -1 $};
\path (-1, 0) node (D) {$ +1 $};
\path[draw] \foreach \i in {A, B, C, D} {(\i) circle[radius=0.3]};
\path[draw, ->, bend left] ($ (A) + (0:0.4) $) to node[pos=0.3, right] {1} ($ (B) + (90:0.4) $);
\path[draw, <-, bend left] ($ (B) + (270:0.4) $) to node[pos=0.7, right] {1} ($ (C) + (0:0.4) $);
\path[draw, ->, bend left] ($ (C) + (180:0.4) $) to node[pos=0.3, left] {1} ($ (D) + (270:0.4) $);
\path[draw, <-, bend left] ($ (D) + (90:0.4) $) to node[pos=0.7, left] {1} ($ (A) + (180:0.4) $);
\end{scope}
\begin{scope}[shift={(8, 0)}]
\path (0, 1) node (A) {$ -1 $};
\path (1, 0) node (B) {$ +1 $};
\path (0, -1) node (C) {$ -1 $};
\path (-1, 0) node (D) {$ +1 $};
\path[draw] \foreach \i in {A, B, C, D} {(\i) circle[radius=0.3]};
\path[draw, <-, bend left] ($ (B) + (270:0.4) $) to node[pos=0.7, right] {1} ($ (C) + (0:0.4) $);
\path[draw, <-, bend left] ($ (D) + (90:0.4) $) to node[pos=0.7, left] {1} ($ (A) + (180:0.4) $);
\end{scope}
\path[draw, thick] (4, -2.5) circle[radius=0.7];
\path[fill] ($ (4, -2.5) + (0:0.7) $) coordinate (X) circle[radius=0.1];
\path[fill] ($ (4, -2.5) + (180:0.7) $) coordinate (Y) circle[radius=0.1];
\path[draw, ultra thick, bend left, <-] ($ (4, -2.5) + (140:0.9) $) to (3, -1.3);
\path[draw, ultra thick, bend right=10, ->] (0, -1.5) to ($ (Y) + (left:0.2) $);
\path[draw, ultra thick, bend left=10, ->] (8, -1.5) to ($ (X) + (right:0.2) $);
\end{tikzpicture}
\caption{This figure depicts the exceptional fiber in the $ A_3 $ case, featuring a 1-parameter family and two singular points. The representations corresponding to the two singular points are drawn, together with the typical representation of the 1-parameter family. Only those arrows with nonzero values are drawn, highlighting that the two singular representations decompose.}
\label{fig:excfiber-partial-A3}
\end{figure}

\subsubsection*{The case of $ D_4 $ with $ θ' = (-1, -1, -1, -1, +2) $}
We find one 1-paramater family of representations and three singular points, depicted in \autoref{fig:excfiber-partial-D4}. In the figure, we have provided an explicit realization of the 1-parameter family. It can alternatively be described as the set of those representations with four pairwise linearly independent vectors $ v_1, …, v_4 ∈ ℂ^2 $ on the arrows from the outer four to the inner vertex, and zeroes in the opposite direction. It is easily verified that the three singular points are indeed distinct. We conclude that the single projective line in the exceptional fiber of $ \M_{θ'} (Q', α') $ corresponds to the middle node of the Dynkin diagram, while the two singular points correspond to the outer three nodes of the Dynkin diagram.

\begin{figure}
\centering
\begin{tikzpicture}
\begin{scope}
\path (2, 0) node (A) {$ -1 $};
\path (0, 2) node (B) {$ -1 $};
\path (-2, 0) node (C) {$ -1 $};
\path (0, -2) node (D) {$ -1 $};
\path (0, 0) node (E) {$ +2 $};
\path[draw, ->] ($ (C) + (0:0.4) $) to node[midway, above] {$ e_1 $} ($ (E) + (180:0.4) $);
\path[draw, ->] ($ (B) + (270:0.4) $) to node[midway, right] {$ e_2 $} ($ (E) + (90:0.4) $);
\path[draw, ->] ($ (A) + (180:0.4) $) to node[midway, below] {$ e_1 + e_2 $} ($ (E) + (0:0.4) $);
\path[draw, ->] ($ (D) + (90:0.4) $) to node[midway, left] {$ e_1 + α e_2 $} ($ (E) + (270:0.4) $);
\path[draw] \foreach \i in {A, B, C, D, E} {(\i) circle[radius=0.3]};
\end{scope}
\begin{scope}[shift={(-5, 0)}]
\path (1.5, 0) node (A) {$ -1 $};
\path (0, 1.5) node (B) {$ -1 $};
\path (-1.5, 0) node (C) {$ -1 $};
\path (0, -1.5) node (D) {$ -1 $};
\path (0, 0) node (E) {$ +2 $};
\path[draw, ->] ($ (C) + (0:0.4) $) to node[midway, above] {$ e_1 $} ($ (E) + (180:0.4) $);
\path[draw, ->] ($ (B) + (270:0.4) $) to node[midway, right] {$ e_1 $} ($ (E) + (90:0.4) $);
\path[draw, ->] ($ (A) + (180:0.4) $) to node[midway, below] {$ e_2 $} ($ (E) + (0:0.4) $);
\path[draw, ->] ($ (D) + (90:0.4) $) to node[midway, left] {$ e_2 $} ($ (E) + (270:0.4) $);
\path[draw] \foreach \i in {A, B, C, D, E} {(\i) circle[radius=0.3]};
\end{scope}
\begin{scope}[shift={(0, 7)}]
\path (1.5, 0) node (A) {$ -1 $};
\path (0, 1.5) node (B) {$ -1 $};
\path (-1.5, 0) node (C) {$ -1 $};
\path (0, -1.5) node (D) {$ -1 $};
\path (0, 0) node (E) {$ +2 $};
\path[draw, ->] ($ (C) + (0:0.4) $) to node[midway, above] {$ e_1 $} ($ (E) + (180:0.4) $);
\path[draw, ->] ($ (B) + (270:0.4) $) to node[midway, right] {$ e_2 $} ($ (E) + (90:0.4) $);
\path[draw, ->] ($ (A) + (180:0.4) $) to node[midway, below] {$ e_1 $} ($ (E) + (0:0.4) $);
\path[draw, ->] ($ (D) + (90:0.4) $) to node[midway, left] {$ e_2 $} ($ (E) + (270:0.4) $);
\path[draw] \foreach \i in {A, B, C, D, E} {(\i) circle[radius=0.3]};
\end{scope}
\begin{scope}[shift={(5, 0)}]
\path (1.5, 0) node (A) {$ -1 $};
\path (0, 1.5) node (B) {$ -1 $};
\path (-1.5, 0) node (C) {$ -1 $};
\path (0, -1.5) node (D) {$ -1 $};
\path (0, 0) node (E) {$ +2 $};
\path[draw, ->] ($ (C) + (0:0.4) $) to node[midway, above] {$ e_1 $} ($ (E) + (180:0.4) $);
\path[draw, ->] ($ (B) + (270:0.4) $) to node[midway, right] {$ e_2 $} ($ (E) + (90:0.4) $);
\path[draw, ->] ($ (A) + (180:0.4) $) to node[midway, below] {$ e_2 $} ($ (E) + (0:0.4) $);
\path[draw, ->] ($ (D) + (90:0.4) $) to node[midway, left] {$ e_1 $} ($ (E) + (270:0.4) $);
\path[draw] \foreach \i in {A, B, C, D, E} {(\i) circle[radius=0.3]};
\end{scope}
\path[draw, thick] (0, 3.7) coordinate (X) circle[radius=0.7];
\path[fill] (X) ++(180:0.7) circle[radius=0.1];
\path[fill] (X) ++(90:0.7) circle[radius=0.1];
\path[fill] (X) ++(0:0.7) circle[radius=0.1];
\path[draw, ultra thick, ->, bend left=10] (-4, 2) to ($ (X) + (180:0.9) $);
\path[draw, ultra thick, ->, bend right=10] (4, 2) to ($ (X) + (0:0.9) $);
\path[draw, ultra thick, ->, bend right] (0.5, 2) to ($ (X) + (290:0.9) $);
\path[draw, ultra thick, ->, bend right=10] (-0.5, 5.5) to ($ (X) + (90:0.9) $);
\end{tikzpicture}
\caption{This figure depicts the exceptional fiber in the $ D_4 $ case, featuring a 1-parameter family and three singular points. The representations corresponding to the three singular points are drawn, together with the typical representation of the 1-parameter family. Only those arrows with nonzero values are drawn, highlighting that the three singular representations decompose.}
\label{fig:excfiber-partial-D4}
\end{figure}

%% file: app_symmetry.tex
\section{Comparison of symmetry groups}
\label{sec:app-symmetry}
In this section, we describe the possible symmetries which can be present in the local quiver. In particular, we prove \autoref{th:result-symmetry}. We start with an arbitrary codimension-2 leaf and look at the local quiver together with the roots assigned to all vertices. It is possible that some of these roots are equal, which gives rise to a permutation action among the vertices, a priori not preserving the quiver structure. Simultaneously, we can consider the automorphisms of the local quiver with vertices labeled by the roots, and this gives rise to an a priori smaller group. In this section, we show that the two groups are in fact equal.

To start our investigation, regard the example of a local quiver with roots depicted in \autoref{fig:symmetry-hypothetical}. In this example, two roots appear twice, and we shall therefore define the naive symmetry group of this labeled quiver to be the permutation group $ S_2 $. However, mirroring over the vertical axis does not preserve the other roots, and we shall therefore define the automorphism group of this labeled local quiver to be the trivial group.

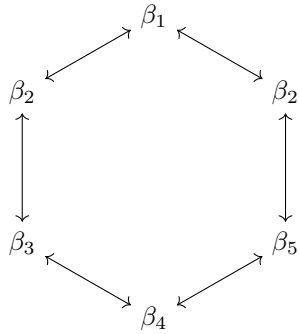
\begin{figure}
\centering
\begin{tikzpicture}
\path (30:2) node (A) {$ β_2 $};
\path (90:2) node (B) {$ β_1 $};
\path (150:2) node (C) {$ β_2 $};
\path (210:2) node (D) {$ β_3 $};
\path (270:2) node (E) {$ β_4 $};
\path (330:2) node (F) {$ β_5 $};
\path[draw, <->] (A) to (B);
\path[draw, <->] (B) to (C);
\path[draw, <->] (C) to (D);
\path[draw, <->] (D) to (E);
\path[draw, <->] (E) to (F);
\path[draw, <->] (F) to (A);
\end{tikzpicture}
\caption{This figure depicts a hypothetical local quiver in which a root appears twice, while the local quiver together with the labels features no symmetry as a whole.}
\label{fig:symmetry-hypothetical}
\end{figure}

\begin{definition}
Let $ (Q, α) $ be a quiver setting with $ α ∈ Σ_{0, 0} $. Let $ L $ be a codimension-2 leaf, given by the roots $ β_1, …, β_k $. Decompose the set of indices into classes $ \{1, …, k\} = I_1 \sqcup … \sqcup I_l $ such that $ β_i = β_j $ if and only if $ i $ and $ j $ are in the same class.
\begin{itemize}
\item The \emph{naive symmetry group} of the labeled local quiver is the product of permutation groups given by
\begin{equation*}
S = \prod_{i = 1}^l S_{I_i} ⊂ S_k.
\end{equation*}
\item The \emph{automorphism group} $ Γ = \Aut(Q', α', (β_1, …, β_k)) $ of the labeled local quiver is the group consisting of those automorphisms of $ (Q', α') $ as directed graph which preserve the roots in the sense that $ β_{σ(i)} = β_i $ for all $ i = 1, …, k $.
\end{itemize}
\end{definition}

Both the naive symmetry group and the automorphism group naturally embed into the permutation group $ S_k $. Indeed, for the naive symmetry group this holds by definition. For the automorphism group, this holds because (apart from the $ A_1 $ case) every pair of vertices is joined by at most one arrow and there are no loops.

Let us regard the case of a local quiver of $ A_1 $ type. This case is special, given that the labeled local quiver has four automorphisms instead of the expected two. However, a symmetry here requires $ (β_1, β_1) = -2 $ while according to \cite[Lemma 7.2]{bellamy-schedler} any imaginary roots appearing twice are necessarily isotropic, in other words $ (β_1, β_1) = 0 $. This shows that the cumbersome $ A_1 $ symmetry cannot appear in the labeled local quiver of any quiver variety.

\begin{lemma}
\label{th:symmetry-agree}
The naive symmetry group and the automorphism group of the labeled local quiver agree: $ S = Γ $.
\end{lemma}

\begin{proof}
We split the proof into cases depending on the type of the leaf. For the $ A_1 $ type, we have already seen that $ S $ is trivial and $ Γ $ is trivial as well. For the $ A_2 $ type, the three roots $ β_1 $, $ β_2 $ and $ β_3 $ are necessarily pairwise distinct since their pairings are required to be $ -1 $ while a self-pairing value is an even integer. For the $ A_3 $ type, there are seven possible labelings of the local quiver, but only the two depicted in \autoref{fig:symmetry-symmetry-possible-A3-1} and \autoref{fig:symmetry-symmetry-possible-A3-2} are possible due to the pairings of the roots and they indeed give rise to automorphisms of the labeled local quiver, hence $ S = Γ $. For the $ A_{≥4} $ type, we observe that any two roots appearing twice have the same adjacent roots by construction of the local quiver, and in consequence both $ S $ and $ Γ $ are trivial.

For the $ D_4 $ type, we observe that the root associated with the central vertex is necessarily distinct from the others due to the adjacencies. The outer four roots may be any combination of equal or non-equal roots. Therefore all symmetries again give rise to automorphisms of the labeled local quiver and we have $ S = Γ $.

For the $ D_{≥5} $ type, we recall that the $ (n-3) $-many central vertices are all pairwise distinct real roots. None of them appear as label on an outer vertex, since that outer vertex would then share the adjacencies with the inner vertex and therefore be adjacent to at least two vertices in total. Furthermore, the two pairs of roots $ \{β_1, β_2\} $ and $ \{β_n, β_{n+1}\} $ are disjoint as well due to non-shared adjacencies. Finally, the only possible symmetries are $ β_1 = β_2 $ and $ β_n = β_{n+1} $, both of which give rise to automorphisms of the labeled local quiver.

For the $ E_6 $ type, recall that the roots at the inner four vertices are pairwise distinct. We observe that they are disjoint from the roots at the outer three vertices due to non-shared adjacencies. Moreover, the outer three roots are pairwise distinct due to non-shared adjacencies. In total, both $ S $ and $ Γ $ are trivial. The same arguments hold for the $ E_7 $ and $ E_8 $ type.
\end{proof}

Next, we show that a symmetry in the local quiver induces monodromy of the resolution. We assume that $ (Q, α) $ is a quiver setting with $ α \in \Sigma_{0, 0} $ and that $ L ⊂ \M (Q, α) $ is a codimension-2 leaf with roots $ β_1, …, β_k $. Let $ θ' \in \mathbb{Z}^{Q_0'} $ be a stability parameter symmetric with respect to all symmetries already present among the $ β_i $.

Regard the local quiver $ (Q', α') $ associated with the leaf $ L $. Let $ σ $ be an automorphism of $ (Q', α') $ which keeps the roots $ β_i $ invariant. This amounts to saying that
\begin{align*}
σ: Q'_0 &\isoto Q'_0, \\
σ: \{a ∈ \overline{Q}_1: i → j\} &\isoto \{a ∈ \overline{Q}_1: σ(i) → σ(j)\}, \\
∀i ∈ Q'_0: \quad β_{σ(i)} &= β_i
\end{align*}

Let us denote by $ \tilde{φ} $ the map $ \tilde{φ}: \Rep(Π_{Q'}, α') → \Rep(Π_{Q'}, α') $ which swaps a given representation around according to $ σ $. More precisely, we put $ \tilde{φ}(ρ)(a) = (-1)^{s_a} ρ(σ^{-1} (a)) $ where the sign $ (-1)^{s_a} $ is $ +1 $ in case $ a $ and $ σ^{-1} (a) $ are both starred or both unstarred, and $ (-1)^{s_a} $ is $ -1 $ if one of them is starred and the other is unstarred.

\begin{remark}
\label{rem:symmetry-diag}
The map $ \tilde{φ}: \Rep(Π_{Q'}, α') → \Rep(Π_{Q'}, α') $ descends to the GIT quotient and to the Mumford by any stability parameter. This is automatic by the construction of GIT and we obtain the commutative diagram
\begin{center}
\begin{tikzcd}
\Rep(π_{Q'}, α') \arrow[r, "\tilde{φ}"] \arrow[d] & \Rep(π_{Q'}, α') \arrow[d] \\
\M_{θ'} (Q', α') \arrow[r, "φ_{θ'}"] \arrow[d] & \M_{θ'} (Q', α') \arrow[d] \\
\M (Q', α') \arrow[r, "φ"] & \M (Q', α')
\end{tikzcd}
\end{center}
\end{remark}

\begin{lemma}
\label{th:app-symmetry-lifting}
Let $ (Q', α') $ be any Kleinian quiver setting and let $ σ $ be an automorphism of $ (Q', α') $. Let $ θ' ∈ ℤ^{Q'_0} $ be any $ σ $-symmetric parameter. Then the map $ φ_{θ'} $ acts by actually flipping a representation according to $ σ $, even on the exceptional fiber.
\end{lemma}

\begin{proof}
Let us regard the flipping map $ \Rep(π_{Q'}, α') \to \Rep(π_{Q'}, α') $. As we can see from the upper commutative square in \autoref{rem:symmetry-diag}, the map $ φ_{θ'} $ is the map induced from flipping the representation. At the same time, since the lower square commutes, the map $ φ_{θ'} $ is necessarily the desired lift $ \tilde{φ} $. This proves the claim.
\end{proof}

%% file: app_monodromy.tex
\section{Description of the monodromy}
\label{sec:app-monodromy}
In this section, we describe the monodromy of the bundle $ π: π^{-1} (L) → L $ for the symplectic codimension-2 leaves $ L $ in case $ α ∈ Σ_{0, 0} $. In particular, we prove \autoref{th:monodromy-fiberbundle}. The idea is that our functor $ F $ from \autoref{sec:main-functor} produces representations which lie above the leaf. As we walk around a closed path in the leaf, we may have the chance to swap the roles of two simples if a symmetry from \autoref{sec:symmetry} is present. We return to a point above the leaf which we can again describe with the help of the functor. This time, the role of the two simples has swapped, which is equivalent to swapping the building plan of the representations. It is precisely the action of the symmetry on the Kleinian exceptional fiber, which we have described in \autoref{th:main-excfiber-action} and elaborated in \autoref{sec:app-excfiber}. Combining these steps, we arrive at an explicit description of the monodromy.

We start with a quiver setting $ (Q, α) $ with $ α ∈ Σ_{0, 0} $ such that $ α $ is indivisible or $ (\gcd(α), p(\gcd(α)^{-1} α)) = (2, 2) $. Let $ L $ be a codimension-2 leaf with labeled local quiver $ (Q', α', (β_1, …, β_k)) $. We perform a case distinction, depending on whether $ α $ is indivisible or not, and depending on the symmetry present in the local quiver. Recall that “generic” and “pseudo-generic” refers to the precise notions listed in \autoref{def:prelim-generic}.

\addtocontents{toc}{\SkipTocEntry}
\subsection*{The case of indivisible $ α $}
By Namikawa's description, we are supposed to pick a symplectic resolution $ π: Y → \M(Q, α) $ and read off the monodromy of the fiber bundle $ π: π^{-1} (L) → L $ over every codimension-2 leaf $ L $. In the case that $ α $ is indivisible, a symplectic resolution is given by $ π: \M_θ (Q, α) → \M(Q, α) $ where $ θ $ is a generic stability parameter. The monodromy of this fiber bundle $ π: π^{-1} (L) → L $ is completely well-defined and is precisely Namikawa's monodromy.

\addtocontents{toc}{\SkipTocEntry}
\subsection*{The case $ (\gcd(α), p(\gcd(α)^{-1} α)) = (2, 2) $}
We claim that in this case the partial resolution $ \M_θ (Q, α) $ with $ θ $ pseudo-generic is already smooth enough to read off the desired monodromy. Let us elaborate this principle in detail. Let $ θ $ be a pseudo-generic stability parameter and $ L $ be a codimension-2 leaf. Then the local quiver $ (Q', α', θ') $ at $ L $ can only have one of the symmetries described in \autoref{fig:symmetry-symmetry-possible}. In the two cases treated in \autoref{sec:app-excfiber-partial}, the variety $ \M_{θ'} (Q', α') $ is only a partial resolution of $ \M(Q', α') $ and the exceptional fiber of the variety $ \M_θ (Q, α) $ above any point in $ L $ consists of a single projective line with two or three singular points, respectively. As we walk around the leaf, the arrangement of this projective line with three singular points is preserved, and we obtain a well-defined action of the fundamental group $ π_1 (L) $ on the components and special points of the exceptional fiber in $ \M_{θ'} (Q', α') $.

In conclusion, in both the indivisible case and the (2, 2) case we have a well-defined notion of “naive monodromy” of $ \M_θ (Q, α) $ over any codimension-2 leaf. It is given by tracking the projective lines and potentially the two or three singular points. The monodromy is by Dynkin automorphism, once an identification between the Dynkin diagram and the projective lines and potentially singular points in the exceptional fiber has been chosen. In \autoref{th:app-monodromy-agree}, we explain that this naive monodromy agrees with Namikawa's monodromy.

\begin{lemma}
\label{th:app-monodromy-agree}
Let $ (Q, α) $ be a quiver setting with $ α ∈ Σ_{0, 0} $ such that $ α $ is indivisible or $ (\gcd(α), p(\gcd(α)^{-1} α)) = (2, 2) $. Let $ θ $ be a pseudo-generic stability parameter and let $ L $ be a codimension-2 leaf of $ \M(Q, α) $. Then the naive monodromy of $ \M_θ (Q, α) → \M(Q, α) $ over $ L $ agrees with Namikawa's monodromy.
\end{lemma}

\begin{proof}
Namikawa's monodromy is the monodromy which we obtain from an actual symplectic resolution. A symplectic resolution $ \widetilde{M_θ (Q, α)} $ is given by further blowing up along the remaining singular locus, which inflates the two or three singular points per $ x ∈ L $ to further projective lines. It is obvious that these points in $ \M_θ (Q, α) $ are interchanged when we move around a loop $ γ: [0, 1] → L $ if and only if the projective lines in $ \widetilde{M_θ (Q, α)} $ are interchanged along the loop $ γ $. In other words, the naive monodromy which we read off from $ \M_θ (Q, α) $ agrees with Namikawa's monodromy.
\end{proof}

The second core ingredient in our identification of the monodromy is the functor $ F $ from \autoref{sec:functor}. Let $ S_1, …, S_k ∈ \Mod_{Π_Q} $ be a collection of simples and $ \tilde S_1, …, \tilde S_k ∈ \Mod_{Π_{Q'}} $ be the vertex simples of the Kleinian local quiver $ (Q', α') $. Then we have the functor
\begin{equation*}
F_{S_1, …, S_k}: \Tw\{\tilde S_1, …, \tilde S_k\} \to \Tw\{S_1, …, S_k\}.
\end{equation*}
We shall now bundle these functors $ F $ together into a single function whose input is a collection $ (S_1, …, S_k) $. We start with a codimension-2 leaf $ L $, given by the isotropic decomposition $ α = m_1 β_1 + … + m_k β_k $. Recall that the representations in $ L $ are precisely those which decompose as $ S_1^{m_1}  ⊕ … ⊕ S_k^{m_k} $ with $ S_i ∈ \M^{\simp} (Q, β_i) $. We write
\begin{equation*}
M ≔ \big(\M^{\simp} (Q, β_1) × … × \M^{\simp} (Q, β_k)\big) \setminus Δ.
\end{equation*}
Here $ Δ ⊂ \prod_{i = 1}^k \M^{\simp} (Q, β_i) $ denotes the subset of tuples where at least two representations are isomorphic. Note that the space $ M $ is not exactly the same as the leaf $ L $. Instead, we have the natural action of $ Γ = \Aut(β_1, …, β_k) $ on $ M $, and the quotient $ M / Γ $ is precisely the leaf $ L $.

Let now $ θ ∈ ℤ^{Q_0} $ be a pseudo-generic stability parameter and $ θ' ∈ ℤ^{Q'_0} $ its localized version. Denote by $ E $ the exceptional fiber of $ π: \M_{θ'} (Q', α') \to \M (Q', α') $. Then we can interpret the family of functors as a continuous function
\begin{equation*}
F: M \times E \to \M_θ (Q, α).
\end{equation*}

\begin{remark}
\label{rem:app-monodromy-invariant}
The group $ Γ $ acts naturally on $ M $ and acts on $ E $ by the lift of the action on $ \M(Q, α) $ thanks to \autoref{th:app-symmetry-lifting}. This way, $ Γ $ also acts on the product $ M × E $ and we notice that the function $ F: M × E → \M_θ (Q, α) $ is in fact $ Γ $-invariant. This can be explained easily as follows: An element $ m ∈ M $ determines the building plan of a representation, and an element $ e ∈ E $ determines the building substance of the representation. When we permute the building plan and simultaneously permute the substance, we obtain the same representation. This is analogous to the fact that when you rotate the building plan of a house and you simultaneously permute the bricks, you obtain the same house. We conclude that $ F $ is $ Γ $ invariant and therefore descends to a map $ F: (M × E) / Γ → \M_θ (Q, α) $.
\end{remark}

\begin{remark}
In the statement of \autoref{th:monodromy-fiberbundle} we make use of the canonical map $ π_1 (L) → Γ $. To describe this map, let $ γ: [0, 1] → L $ be a closed path and take a lift $ \tilde{γ}: [0, 1] → M $. Then determine the $ g ∈ Γ $ such that $ \tilde{γ}(1) = g \tilde{γ}(0) $. Simply speaking, the map $ π_1 (L) → Γ $ forgets the entire trajectory of the path, and remembers only which simples are swapped along the way.
\end{remark}

We are now ready to prove \autoref{th:monodromy-fiberbundle}, which claims that $ F: (M × E) / Γ → π^{-1} (L) $ is an isomorphism. This identifies the fiber bundle $ π: π^{-1} (L) → L $ and tells us that the desired monodromy is nothing else than the action of $ Γ $ on $ E $.

\begin{proof}[Proof of \autoref{th:monodromy-fiberbundle}]
Let us explain all three statements after each other. Regard the first statement. The projection $ M × E → M $ is a trivial fiber bundle with fiber $ E $. The action by the finite group $ Γ $ is free, therefore its quotient is still a fiber bundle with fiber $ E $.

Regard the second statement. Let us observe that $ F $ indeed maps to $ π^{-1} (L) $. Indeed, let $ m = (S_1, …, S_k) ∈ M $ and let $ e ∈ E $. Then $ F(m, e) $ is a $ θ $-polystable representation by \autoref{th:main-functor-stability} which semisimplifies to $ S_1^{m_1} ⊕ … ⊕ S_k^{m_k} $. Therefore we have $ F(m, e) ∈ π^{-1} (L) $. Moreover the diagram is commutative as claimed.

Let us explain that the map $ F: (M × E) / Γ → π^{-1} (L) $ is an isomorphism. It suffices to check the bijectivity fiberwise. Recall that the functor $ F_{S_1, …, S_k}: \Tw\{\tilde S_1, …, \tilde S_k\} → \Tw\{S_1, …, S_k\} $ is a strict $ A_∞ $-equivalence. This makes $ F: E → π^{-1} ([S_1, …, S_k]) $ injective. Since any $ θ $-polystable representation in $ π^{-1} ([S_1, …, S_k]) $ lies in $ \Tw\{S_1, …, S_k\} $, the map $ F: E → π^{-1} ([S_1, …, S_k]) $ is also surjective. This proves fiberwise bijectivity and finishes the second statement.

Regard the third statement. Let $ γ: [0, 1] → L $ be a closed path and $ \tilde{γ}: [0, 1] → L $ is its lift to $ L $ with $ \tilde{γ} (1) = g \tilde{γ} (0) $. We need to show that the monodromy action of $ γ $ on $ π^{-1} (γ(0)) $ is equal to the action of $ g $ on $ E $. Thanks to the second statement, this is immediate, but we spell out the instructive detail. Write $ γ(0) = F(\tilde{γ} (0), e) $ with $ m ∈ M $ and $ e ∈ E $. Then the lift of $ γ $ to the exceptional fiber $ π^{-1} (L) $ can be described explicitly by $ ρ(t) = F(\tilde{γ} (t), e) $. In particular, we have
\begin{equation*}
ρ(1) = F(\tilde{γ} (1), e) = F(g \tilde{γ} (0), e) = F(\tilde{γ} (0), g^{-1} e).
\end{equation*}
In other words, while $ ρ(0) $ lies in the projective line or singular point given by $ e $, the representation $ ρ(1) $ lies in the projective line or singular point given by $ g^{-1} e $. This proves the third statement and finishes the proof.
\end{proof}

%% file: app_products.tex
\section{Products and symmetric products}
\label{sec:app-products}
In the present section, we show how to determine the Namikawa-Weyl groups of symmetric and direct products. Together with the analysis of the Namikawa-Weyl groups of the individual quiver varieties $ \M(Q, α_i) $ in the previous section, this lets us determine the Namikawa-Weyl groups of arbitrary quiver varieties that have a symplectic resolution. For brevity, we denote the Namikawa-Weyl group of $ X $ by $ W(X) $.

\begin{lemma}
Let $ X_1, …, X_k $ be symplectic singularities with good $ ℂ^* $-action. Then $ W(X_1 × … × W_k) = W(X_1) × … × W(X_k) $.
\end{lemma}

\begin{proof}
Choose symplectic resolutions $ π_i: Y_i → X_i $. Then the codimension-2 leaves of $ X_1 × … × X_k $ are given by $ X_1^{\reg} × … × L × … × X_k^{\reg} $ where the $ L $ is at the $ i $-th index and $ L $ runs through the set of codimension-2 leaves of $ X_i $. A symplectic resolution $ π $ of the product variety is given by the product of the individual resolutions $ π_i $. We recall that all individual resolutions $ π_i $ are trivial fiber bundles over their smooth locus. To read off the monodromy, we regard the fiber bundle
\begin{equation*}
π: X_1^{\reg} × … × π_1^{-1} (L) × … × X_k^{\reg} → X_1^{\reg} × … × L × … × X_k^{\reg}.
\end{equation*}
It is the direct product of the trivial bundles $ π_i: X_i^{\reg} → X_i^{\reg} $ and of $ π_i: π_i^{-1} (L) → L $, and therefore has the same monodromy as $ L $. This proves the claim.
\end{proof}

In the lemma below we perform a similar analysis for the symmetric products. It is slightly trickier than the case of direct products, but still relies only on classical observations.

\begin{lemma}
Let $ X $ be a symplectic singularity with good $ ℂ^* $-action. Then $ W(S^n X) = W(X) $.
\end{lemma}

\begin{proof}
We differentiate the case whether $ X $ has dimension $ 2 $ or dimension at least $ 4 $.

Let us start with the case that $ X $ has dimension $ 2 $. Then $ X $ is in fact a Kleinian singularity. The variety $ S^n X $ has two types of codimension-2 leaves. The first is the leaf $ L $ which contains the points $ (x_1, …, x_n) ∈ S^n X $ where $ x_1 = 0 $ and $ x_2, …, x_n ∈ X^{\reg} $ are pairwise distinct. A resolution of a neighborhood of such a leaf is given by the product $ \tilde{X} × L → X × L $. In particular, the monodromy is the same as the monodromy of the codimension-2 leaf $ L ⊂ X $.

The second type of codimension-2 leaves can be described as follows. Let $ 1 ≤ i < j ≤ n $, then the leaf $ L_{ij} $ consists of points $ (x_1, …, x_n) ∈ S^n $ where $ x_i = x_j $ and all other entries are pairwise distinct and lie in the smooth locus of $ X $. Locally around this leaf, the singularity looks like $ (ℂ^2 × ℂ^2) / C_2 $, which is isomorphic to $ ℂ^2 × (ℂ^2 / C_2) $ and thus of $ A_1 $ type. This gives no additional monodromy.

Let us now treat the case that $ X $ has dimension $ ≥ 4 $. Similar to the case of dimension $ 2 $, every codimension-2 leaf of $ X $ now reappears as codimension-2 leaf for $ S^n X $, with the same monodromy. The second type of leaves does not appear, given that any connected component of the open locus already has dimension $ 4 $.

Finally, we conclude that $ S^n X $ has the same Namikawa-Weyl group as $ X $ in both cases. This finishes the proof.
\end{proof}

%% file: bibliography.bib
@misc{bellamy-schedler,
author = {{Bellamy}, G. and {Schedler}, T.},
title = "{Symplectic resolutions of Quiver varieties and character varieties}",
note  = "\url{math.AG/1602.00164}",
year = 2016,
}

@article{namikawa-I,
author={{Namikawa}, Y.},
title="Poisson deformations of affine symplectic varieties",
journal="Duke Math.\ J.",
volume=156,
number=1,
year=2011,
pages="51–85"
}

@article{namikawa-II,
author={{Namikawa}, Y.},
title="Poisson deformations of affine symplectic varieties II",
journal="Kyoto J.\ Math.",
volume=50,
number=4,
year=2010,
pages="727–752"
}

@article{kaledin-stratification,
author={{Kaledin}, D.},
title="Symplectic singularities from the Poisson point of view",
journal="J.\ Reine Angew.\ Math.",
volume=600,
year=2006,
pages="135–156"
}

@article{cb-geometry-moment,
author={{Crawley-Boevey}, W.},
title="Geometry of the moment map for representations of quivers",
journal="Compositio Math.",
volume=126,
number=3,
year=2001,
pages="257–293"
}

@article{cb-kleinian,
author={{Crawley-Boevey}, W.},
title="On the exceptional fibres of Kleinian singularities",
journal="Amer.\ J.\ Math.",
volume=122,
number=5,
year=2000,
pages="1027–1037"
}

@article{bellamy-cm,
author={{Bellamy}, G.},
title="Counting resolutions of symplectic quotient singularities",
journal="Compos.\ Math.",
volume=152,
number=1,
year=2016,
pages="99–114"
}

@article {Beauville,
    AUTHOR = {Beauville, Arnaud},
     TITLE = {Symplectic singularities},
   JOURNAL = {Invent. Math.},
  FJOURNAL = {Inventiones Mathematicae},
    VOLUME = {139},
      YEAR = {2000},
    NUMBER = {3},
     PAGES = {541--549},
}

@article {Nakajima,
    AUTHOR = {Nakajima, Hiraku},
     TITLE = {Instantons on {ALE} spaces, quiver varieties, and
              {K}ac-{M}oody algebras},
   JOURNAL = {Duke Math. J.},
  FJOURNAL = {Duke Mathematical Journal},
    VOLUME = {76},
      YEAR = {1994},
    NUMBER = {2},
     PAGES = {365--416},
}

@article {Wu,
    AUTHOR = {Wu, Yaochen},
     TITLE = {Namikawa-{W}eyl groups of affinizations of smooth {N}akajima
              quiver varieties},
   JOURNAL = {Represent. Theory},
  FJOURNAL = {Representation Theory. An Electronic Journal of the American
              Mathematical Society},
    VOLUME = {27},
      YEAR = {2023},
     PAGES = {734--765},
}

@article {CrawleyBoevey-Kimura,
    AUTHOR = {Crawley-Boevey, William and Kimura, Yuta},
     TITLE = {On deformed preprojective algebras},
   JOURNAL = {J. Pure Appl. Algebra},
  FJOURNAL = {Journal of Pure and Applied Algebra},
    VOLUME = {226},
      YEAR = {2022},
    NUMBER = {12},
     PAGES = {Paper No. 107130, 22},
}

@misc{Davison,
      title={Purity and 2-Calabi-Yau categories}, 
      author={Ben Davison},
      year={2023},
      eprint={2106.07692},
      archivePrefix={arXiv},
      primaryClass={math.AG}
}

@book {Bocklandt-book,
    AUTHOR = {Bocklandt, Raf},
     TITLE = {A gentle introduction to homological mirror symmetry},
    SERIES = {London Mathematical Society Student Texts},
    VOLUME = {99},
 PUBLISHER = {Cambridge University Press, Cambridge},
      YEAR = {2021},
     PAGES = {xi+390},
      ISBN = {978-1-108-48350-6; 978-1-108-72875-1},
   MRCLASS = {14J33 (00A79 14F08 16E99 16G20 18G70)},
  MRNUMBER = {4297810},
       DOI = {10.1017/9781108692458},
       URL = {https://doi-org.proxy.uba.uva.nl/10.1017/9781108692458},
}
